\documentclass{amsart}

\usepackage{amsmath, amsthm, amssymb}
\usepackage{hyperref} 
\usepackage[usenames,dvipsnames]{xcolor}

\usepackage{verbatim}
\usepackage{tikz}
\usepackage{graphicx}
\usepackage{tikz-cd}
\usepackage{enumerate}

\newtheorem{thm}{Theorem}[section]
\newtheorem{lemma}[thm]{Lemma}%[section]
\newtheorem{prop}[thm]{Proposition}%[section]
\newtheorem{cor}[thm]{Corollary}%[section]
\newtheorem{claim}[thm]{Claim}
\newtheorem{question}[thm]{Question}

\theoremstyle{remark}
\newtheorem{rmk}[thm]{Remark}%[section]

\theoremstyle{definition}
\newtheorem{definition}[thm]{Definition}%[section]
\newtheorem{example}[thm]{Example}

\numberwithin{equation}{section}
\newcommand{\ben}{\begin{enumerate}}
\newcommand{\een}{\end{enumerate}}
\newcommand{\RR}{\mathbb{R}}

\newcommand{\dist}{\textnormal{dist}}
\newcommand{\diam}{\textnormal{diam }}
\newcommand{\R}{\mathbb{R}}

\begin{document}

\title{Rectifiability of planes and Alberti representations} 
\author{Guy C. David}
\address{Courant Institute of Mathematical Sciences, New York University, New York, NY 10012}
\email{guydavid@math.nyu.edu}

\author{Bruce Kleiner}
\address{Courant Institute of Mathematical Sciences, New York University, New York, NY 10012}
\email{bkleiner@cims.nyu.edu}

\begin{abstract}
We study metric measure spaces that have quantitative topological control, as well as a weak form of differentiable structure. In particular, let $X$ be a pointwise
doubling metric measure space. Let $U$ be a Borel subset on which the blowups of $X$ are topological planes. We show that $U$ can admit at most $2$ independent Alberti representations. Furthermore, if $U$ admits $2$ Alberti representations, then the restriction of the measure to $U$ is $2$-rectifiable. This is a partial answer to the case $n=2$ of a question of the second author and Schioppa.
\end{abstract}

\thanks{B.K. was supported by a Simons Fellowship, a Simons collaboration
grant,  and NSF grants DMS-1405899, DMS-1406394.}
\date{\today}
\maketitle
\setcounter{tocdepth}{1}
\tableofcontents
\section{Introduction}

In the last two decades, there has been tremendous progress in developing notions of differentiability in the context of metric measure spaces. One strand of this work, beginning with Cheeger \cite{Ch99}, focuses on PI spaces -- metric measure spaces that are doubling and support a Poincar\'e inequality in the sense of \cite{HK98} -- and, more generally, Lipschitz differentiability spaces -- those equipped with a differentiable structure in the form of a certain measurable (co)tangent bundle.
Other strands focus on differentiability via the presence of independent $1$-rectifiable structures, as in the work of Bate-Alberti-Csornyei-Preiss \cite{Al93,ACP05,Ba15}, or the existence of independent derivations mimicking the functional analytic properties of partial derivatives, as in the work of Weaver \cite{We00}.

By now, the different definitions of Cheeger, Weaver, and Bate-Alberti-Csornyei-Preiss have been shown to be closely related \cite{Ba15, Sc16, Sc14}, 
and a variety of applications of differentiability have been given, e.g., in geometric group theory, in embedding theory, and in Sobolev and geometric function theory.

Nonetheless, some issues are still not well understood, in particular the relation between analytical structure and quantitative topology.  A particularly intriguing question along this line is, loosely speaking: What kind of differentiable structure can live on a topological manifold?  To make this precise using a condition from quantitative topology, we impose one of the following on our metric space $X$:
\begin{enumerate}
\renewcommand{\theenumi}{\alph{enumi}}
\item\label{LCcondition} $X$ is linearly contractible, i.e., for some $\lambda>0$, each ball $B(x,r)$ is contractible in the ball $B(x,\lambda r)$. % Semmes.
\item\label{SScondition} $X$ is a self-similar topological $n$-manifold, or more generally,  all blowups (pointed Gromov-Hausdorff limits of rescalings) of $X$ are topological $n$-manifolds.
\end{enumerate}
Both of these are widely used strengthenings of the qualitative hypothesis that $X$ is a manifold; condition \eqref{LCcondition} was used heavily by Semmes \cite{Se96}.

One may then ask (see \cite{KS15} for discussion of this and other related questions):

\begin{question}
\label{que_manifold_examples}
Let $(X,d,\mu)$ be a PI space %(a metric measure space that is doubling and supports a Poincar\'e inequality in the sense of \cite{HK98}),
and assume $X$ is a topological $n$-manifold satisfying \eqref{LCcondition} or \eqref{SScondition} above.  What can be said about the dimension of the differentiable structure,  the Hausdorff dimension, or the structure of blow-ups of $(X,d,\mu)$?% (Gromov-Hausdorff limits of pointed rescalings of $(X,\mu)$)? 
\end{question}

Note that the only known examples of PI spaces as above are sub-Riemannian manifolds or variations on them; in particular, in all known examples, blowups at generic points are bilipschitz homeomorphic to Carnot groups.  

A special case of our main result gives some restrictions in the $2$-dimensional case:

\begin{thm}
\label{thm_intro_special_case}
Suppose $(X,d,\mu)$ is a PI space, and all
blowups of $X$ are homeomorphic to $\R^2$.
Then the tangent bundle has dimension at most $2$, and if $U\subset X$ is a Borel set on which the tangent bundle has dimension $2$, then $\mu|_U$  is $2$-rectifiable.
\end{thm}

We point out that apart from the above theorem, Question \ref{que_manifold_examples} is wide open.  In the $n=2$  case, we do not know if the PI space $X$ could have a $1$-dimensional tangent bundle, while in the $n\geq 3$ case we know of no nontrivial restriction on either the Hausdorff dimension or the dimension of the differentiable structure. For instance, when $n=3$, it is not known if $X$ could be Ahlfors $100$-regular, or if it could have a tangent bundle of dimension $100$ or dimension $1$. 

We now state our main theorem. Before doing so, we remark that it generalizes the special case stated above in a few ways. For one, we do not require that $X$ is a PI space, or even the weaker condition that $X$ is a Lipschitz differentiability space, but just that the measure $\mu$ on $X$ supports at least two independent \textit{Alberti representations}. (Alberti representations are certain decompositions of $\mu$ as superpositions of measures supported on one-dimensional curve fragments, which we define precisely below. They were introduced in metric measure spaces by Bate \cite{Ba15}, building on work of Alberti \cite{Al93} and Alberti-Cs\"ornyei-Preiss \cite{ACP05}.) For another, we do not require that $X$ itself is a topological plane, but only that its blowups are.

\begin{thm}\label{thm:main}
Let $(X,d,\mu)$ be a pointwise doubling metric measure space. Let $U\subset X$ be a Borel subset such that, for $\mu$-a.e. $x\in U$, each blowup of $X$ at $x$ is homeomorphic to $\RR^2$.

If $\mu|_U$ has $n$ $\phi$-independent Alberti representations for some Lipschitz $\phi\colon X\rightarrow\RR^n$, then $n\leq 2$, and equality holds only if $\mu|_U$ is $2$-rectifiable. 
\end{thm}

Recall that a measure $\mu$ on a metric space $X$ is called $m$-rectifiable if there are countably many compact sets $E_i \subset \RR^m$ and Lipschitz mappings $g_i:E_i\rightarrow X$ such that $\mu(X\setminus \cup g_i(E_i)) = 0$.  All the remaining terminology used in Theorem \ref{thm:main} will be defined in Section \ref{sec:prelims}.

In Theorem \ref{thm:main}, %and Corollary \ref{cor:main}
one may replace the assumption that there exist $n$ $\phi$-independent Alberti representations with the assumption that there exist $n$ linearly independent Weaver derivations. (See \cite{We00} or \cite[Section 13]{He07}, for an explanation of Weaver derivations.) Up to decomposing $U$, these assumptions are equivalent by the work of Schioppa, in particular by Theorem 3.24 and Corollary 3.93 of \cite{Sc16}.

By the work of Bate \cite{Ba15}, Theorem \ref{thm:main} %and Corollary \ref{cor:main}
also applies if the assumption that there exist $n$ $\phi$-independent Alberti representations is replaced by the assumption that $X$ is a Lipschitz differentiability space and $U$ is an $n$-dimensional chart in $X$. (See \cite{Ba15} and \cite{Ch99} for more about Lipschitz differentiability spaces.)

\begin{rmk}
On its own, $m$-rectifiability of a measure $\mu$, as defined above, does not imply that $\mu$ is absolutely continuous with respect to $m$-dimensional Hausdorff measure (see \cite{GKS10}). Nonetheless, it follows from \cite{ACP05} and \cite[Corollary 6.10]{Ba15} that under the assumption of equality in Theorem \ref{thm:main}, the measure $\mu$ must be absolutely continuous with respect to two-dimensional Hausdorff measure. (That a similar implication works also in higher dimensions follows from the recent work \cite{DR16, DMR16} or, alternatively, from an announced result of Cs\"ornyei-Jones.) Note that in Theorem \ref{thm:main} we do not assume that two-dimensional Hausdorff measure is $\sigma$-finite on $X$, or indeed anything about the Hausdorff dimension of $X$.
%\red{Since ACP never appeared, we could cite the de-Phillipis-Rindler paper.  Could also mention Csornyei-Jones, which also never appeared.}
\end{rmk}

One specific application of Theorem \ref{thm:main} is when the space $X$ itself is a linearly locally contractible topological surface. This fits neatly into the theme, described above, about the interaction between analysis and (quantitative) topology.

\begin{cor}\label{cor:main}
Let $X$ be a metrically doubling, LLC, topological surface. Let $\mu$ be a pointwise doubling measure on $X$ such that all porous sets in $X$ have $\mu$-measure zero, and let $U$ be a Borel subset of $X$.

If $\mu|_U$ has $n$ $\phi$-independent Alberti representations for some Lipschitz $\phi\colon X\rightarrow\RR^n$, then $n\leq 2$, and equality holds only if $\mu|_U$ is $2$-rectifiable. 
\end{cor}

\vspace{\baselineskip}
\subsection*{Rectifiability and related prior work}
One can view Theorem \ref{thm:main} as providing a sufficient condition for a measure $\mu$ on a metric space $X$ to be rectifiable. For measures defined on subsets of Euclidean space, there is a huge literature of such results, most famously the work of Preiss \cite{Pr87} which applies under density assumptions on $\mu$. There is also a program to understand the more quantitative notion of uniform rectifiability for measures on Euclidean space, which often entails assumptions of quantitative topology \cite{DS93, DS00}.

In the case of abstract metric measure spaces, there are fewer known sufficient conditions for rectifiability. Closer to our present setting, Bate and Li \cite{BL15} also make a connection between the existence of Alberti representations in a metric measure space and the rectifiability of that space. They assume nothing about the topology of $X$ but rather impose density assumptions on the measure $\mu$. More specifically, Theorem 1.2 of \cite{BL15} shows the following: Suppose $(X,d,\mu)$ is a metric measure space such that $\mu$ has positive and finite upper and lower $n$-dimensional densities almost everywhere. Then $\mu$ is $n$-rectifiable if and only if it admits a measurable decomposition into sets $U_i$ such that each $\mu|_{U_i}$ supports $n$ independent Alberti representations.

Theorem \ref{thm:main} above says that, under certain quantitative topological assumptions on a space $(X,d,\mu)$ supporting two independent Alberti representations, one can conclude $2$-rectifiability of the measure $\mu$ without any \textit{a priori} density assumptions and, in particular, without any assumption on the Hausdorff dimension of the space.

With strong assumptions on both the measure $\mu$ and the quantitative topology of $X$, much more can be said. For example, in \cite{Se96}, Semmes shows that an Ahlfors $n$-regular, linearly locally contractible $n$-manifold must be a PI space, and in \cite{GCD16}, it is shown that such manifolds are locally uniformly $n$-rectifiable.

Specializing further to the the $2$-dimensional case, in \cite{BK02} it is shown that one can even achieve global parametrizations: an Ahlfors $2$-regular, linearly locally contractible $2$-sphere $X$ is quasisymmetrically equivalent to the standard $2$-sphere. This result implies both that $X$ is a PI space and that $2$-dimensional Hausdorff measure on $X$ is $2$-rectifiable. Related parametrization results for other topological surfaces appear in \cite{Wi08, Wi10}.

\subsection*{Remarks on the proof of Theorem \ref{thm:main}}
Let us briefly discuss the proof that equality $n=2$ in Theorem \ref{thm:main} implies $2$-rectifiability, which is the more difficult part of the theorem. The proof consists of three main ingredients.

In the first ingredient, we use the fact that generic blowups of $\phi$ are certain ``model mappings'', namely, Lipschitz quotient mappings from doubling, linearly contractible planes to $\RR^2$ (Proposition \ref{prop:LQblowup}). (For the definitions of these terms, see Sections \ref{sec:prelims} and \ref{sec:LQ}.) 

Thus, up to decomposing $U$, we may assume there is a scale $r_0>0$ such that at almost every point in $U$, $\phi$ is uniformly close to these model mappings, in the Gromov-Hausdorff sense, at scales below $r_0$ (Lemma \ref{lem:rescaling}). By decomposing further, we may also assume that the model mappings have uniform constants and that $\mu$ is essentially a doubling measure on $U$, in particular, that porous subsets of $U$ have measure zero.

For the second ingredient, we study the geometry of the model mappings. The study of Lipschitz quotient mappings of the plane initiated in \cite{BJLPS99, JLPS00} already shows that the model mappings are discrete and open, i.e., are branched coverings (Propositions \ref{prop:LQdiscrete} and \ref{prop:vaisala}). More quantitatively, we prove by a compactness argument that each model mapping is bilipschitz on a sub-ball of quantitative size in every ball of its domain (Proposition \ref{prop:subball}).

The final ingredient in the proof is the following stability property of the model mappings: If two model mappings are sufficiently close in the Gromov-Hausdorff sense and one is bilipschitz on a ball of radius $R$, then the other is bi-Lipschitz on a ball of radius $cR$, for some controlled constant $c>0$ (Lemma \ref{lem:degree1}). This means that if $\phi$ is close to a bilipschitz model mapping at some point and scale, then this property persists under further magnification at this point (Lemma \ref{lem:iteration}).

Combining these three ingredients shows that the set of points in $U$ at which $\phi|_U$ is not locally bilipschitz is in fact porous (or rather, $\sigma$-porous) in $U$, and hence has measure zero.

\subsection*{Structure of the paper}
The outline of the paper is as follows. Section \ref{sec:prelims} contains notation and preliminary definitions. Section \ref{sec:LQ} contains the definition and properties of Lipschitz quotient mappings. Section \ref{sec:close} contains some preliminary definitions related to Gromov-Hausdorff convergence of metric spaces, and Sections \ref{sec:topology} and \ref{sec:compactness} contain some further quantitative topological facts about Lipschitz quotient mappings on planes. In Section \ref{sec:proof} we prove Theorem \ref{thm:main} and Corollary \ref{cor:main}, and in Section \ref{sec:examples} we provide some relevant examples where the results fail under relaxed assumptions. Section \ref{sec:appendix} is an appendix which shows that a space whose blowups are all planes actually satisfies a more quantitative condition on its blowups. This used in the proof of Theorem \ref{thm:main}, but its proof is a modification of fairly standard ideas in the literature and so is relegated until the end.

\section{Notation and preliminaries}\label{sec:prelims}
Throughout this paper, we consider only complete, separable metric spaces and locally finite Borel regular measures. %yes?

\subsection{Metric space notions}

We write
$$ B(x,r) = B_X(x,r) = \{y\in X: d(y,x)<r\}$$
for the open ball in a metric space $X$, and we write
$$ \overline{B}(x,r) = \overline{B_X}(x,r) = \{y\in X: d(y,x)\leq r\},$$
which need not be the closure of $B(x,r)$. We also consider closed annuli, which we write as
$$ A(x,r,R) = \{y\in X: r\leq d(y,x)\leq R\} $$

A metric space is called \textit{proper} if $\overline{B}(x,r)$ is compact in $X$ for each $x\in X$ and $r>0$.

Given a metric space $(X,d)$ and $\lambda>0$, we will write $\lambda X$ for the metric space $(X,\lambda d)$.

\begin{definition}
A metric space $X$ is called \textit{metrically doubling} if there is a constant $D\geq 0$ such that each ball in $X$ can be covered by at most $D$ balls of half the radius. If we wish to emphasize the constant $D$, we will call $X$ metrically $D$-doubling.
\end{definition}

If $X$ is complete and metrically doubling, then it is proper.

\begin{definition}
Let $X$ be a metric space and $S\subset X$ a subset. We say that $S$ is \textit{porous} if, for all $x\in S$, there is a constant $\eta>0$ and a sequence $x_n\rightarrow x$ in $X$ such that
$$ B(x_n, \eta d(x_n,x)) \cap S = \emptyset.$$
\end{definition}

We now introduce some terms from quantitative topology, including the term linearly locally contractible used in Corollary \ref{cor:main}.
\begin{definition}\label{def:LLCALC}
Let $X$ be a metric space.
\begin{itemize}
\item We call $X$ \textit{linearly locally contractible (LLC)} if there is a radius $r_0>0$ and a constant $A\geq 1$ such that each metric ball $B(x,r)\subset X$ with $r<r_0$ is contractible inside $B(x,Ar)$. If one can take $r_0=\infty$, we call the space \textit{linearly contractible (LC)}, or $A$-LC if we wish to emphasize the constant.

\item We call $X$ \textit{annularly linearly connected} if there is a constant $\lambda\geq 1$ such that, for all $p\in X$ and $r\in (0, \diam(X)]$, any two points $x,y\in A(p, r, 2r)$ can be joined by a continuum in $A(p,r/\lambda, 2\lambda r)$. We abbreviate this condition as ALC, or $\lambda$-ALC to emphasize the constant. 
\end{itemize}
\end{definition}

We will make use of the following relationship between the above two notions, which is a minor modification of facts found in the literature.
\begin{lemma}\label{lem:ALCLC}
If a complete, metrically doubling space $X$ is annularly linearly connected with constant $\lambda$ and homeomorphic to $\RR^2$, then $X$ is linearly contractible, with constant $A=A(\lambda)$ depending only on $\lambda$.
\end{lemma}
\begin{proof}
Lemma 5.2 of \cite{Ki16} shows that $X$ must satisfy two conditions known as ``$\text{LLC}_1$'' and ``$\text{LLC}_2$'', with constants depending only on $\lambda$. The argument in Lemma 2.5 of \cite{BK02} then shows that $X$ must be linearly contractible with constant depending only on $\lambda$. (Note that, since $X$ is homeomorphic to $\RR^2$, there is no need to restrict to a bounded subset as in the proof of that lemma.)
\end{proof}

We will use the notion of simultaneous pointed Gromov-Hausdorff convergence of spaces and functions. Namely, we will consider triples $(X,p,\phi)$, where $X$ is a metric space, $p\in X$ is a base point, and $\phi:X\rightarrow \RR^k$ is a Lipschitz function. This type of convergence is explained in detail in a number of places. See, for example, Chapter 8 of \cite{DS97}, \cite{Ke04}, \cite{GCD15}, or \cite{CKS15}.

If $(X_n, p_n)$ is a sequence of metrically $D$-doubling spaces, then it has a subsequence which converges in the pointed Gromov-Hausdorff sense to a metrically $D$-doubling space. If furthermore $f_n:X_n\rightarrow \RR$ are all $L$-Lipschitz functions, for some fixed $L$, then $\{(X_n, p_n, f_n)\}$ has a subsequence converging to a triple $(X,p,f)$ for which $f$ is $L$-Lipschitz.

\begin{definition}
Let $X$ be a metric space, $p\in X$ a point, and $\{\lambda_k\}$ is a sequence of positive real numbers tending to zero. If the sequence
$$\left\{ \left(\lambda_k^{-1}X, p\right)  \right\} $$
converges in the pointed Gromov-Hausdorff sense to a space $(\hat{X},\hat{p})$, then $(\hat{X}, \hat{p})$ is called a \textit{blowup} of $X$ at $p$.

Let $\phi:X\rightarrow\RR^2$ be a Lipschitz function. If the sequence
$$\left\{ \left(\lambda_k^{-1}X, p, \lambda_k^{-1}(\phi - \phi(p))\right)  \right\}$$
converges to a triple $(\hat{X}, \hat{p}, \hat{\phi})$, then $(\hat{X}, \hat{p}, \hat{\phi})$ is called a \textit{blowup} of $(X,p,\phi)$. In this case $(\hat{X}, \hat{p})$ will be a blowup of $(X,p)$.
\end{definition}

By our previous remarks, a metrically doubling space admits blowups at each of its points, as does a metrically doubling space together with a Lipschitz function. For the definition of blowups at almost every point of a pointwise doubling metric measure space, as used in Theorem \ref{thm:main}, see Remark \ref{rmk:ptwiseblowup}.

\subsection{Metric measure space notions}
Recall our standing assumption that our metric measure spaces are complete, separable, and Borel regular.

If $(X,d,\mu)$ is a metric measure space and $U\subset X$ is a measurable subset, then we write $\mu|_U$ for the measure defined by
$$ \mu|_U(E) = \mu(E\cap U).$$

The definition of ``metrically doubling'' for metric spaces has already appeared; now we introduce the related concept for metric measure spaces.

\begin{definition}
A metric measure space $(X,d,\mu)$ is called \textit{doubling} if there is a constant $C\geq 0$ such that
$$ \mu(B(x,r))\leq C\mu(B(x,2r) $$
for all $x\in X$ and $r>0$.
\end{definition}

If a metric measure space is doubling, then the underlying metric space is metrically doubling (see \cite{He01}). Of course, a metrically doubling space may carry a specific measure $\mu$ which is not doubling.

\begin{definition}
A metric measure space $(X,d,\mu)$ is called \textit{pointwise doubling at $x\in X$} if
$$ \limsup_{r\searrow 0} \frac{\mu(B(x,2r))}{\mu(B(x,r))} < \infty.$$ 
We call $(X,d,\mu)$ \textit{pointwise doubling} if it is pointwise doubling at $\mu$-a.e. $x\in X$.
\end{definition}

\begin{definition}
We say that $(X,d,\mu)$ is \textit{$(C,R)$-uniformly pointwise doubling at $x\in X$} if
\begin{equation}\label{eqn:unifptwise}
\mu(B(x,r)) \leq C \mu(B(x,r/2)) \hspace{0.5in} \text{for all } r< R
\end{equation}

If, for some $C\geq 1$ and $R>0$, the space $(X,d,\mu)$ is $(C,R)$-uniformly pointwise doubling at $\mu$-a.e. $x\in X$, we call $(X,d,\mu)$ \textit{uniformly pointwise doubling}.

A subset $A\subset X$ is called \textit{$(C,R)$-uniformly pointwise doubling} if $\mu$ is $(C,R)$-uniformly pointwise doubling at $x$ for all $x\in A$. (Note that we ask that \eqref{eqn:unifptwise} holds for balls in $X$ centered at points of $A$, not that $(A,d,\mu)$ is uniformly pointwise doubling at all $x\in A$.)
\end{definition}

We note that the Lebesgue density theorem applies to pointwise doubling measures; see Section 3.4 of \cite{HKST15}. From this it follows immediately that if $(X,d,\mu)$ is pointwise doubling and $U\subset X$ is Borel, then $(U,d,\mu|_U)$ is pointwise doubling.

\begin{rmk}\label{rmk:ptwiseporous}
If $\mu$ is a doubling measure on $X$, then every porous set in $X$ has $\mu$-measure zero. However, it is not true that every porous set in a pointwise doubling metric measure space $(X,d,\mu)$ must have measure zero. (For example, take $\RR^2$ equipped with the measure $\mu$ which is the restriction of $\mathcal{H}^1$ to a single line.) However, the following fact is immediate from the Lebesgue density theorem: If $A\subset X$ is a uniformly pointwise doubling subset, and $S\subset A$ is porous as a subset of the metric space $(A,d)$, then $S$ has measure zero.
\end{rmk}

The following facts about pointwise doubling spaces combine Lemmas 2.2 and 2.3 of \cite{BL15_RNP} (see also \cite{BS13} and \cite{Ba15}).
\begin{lemma}\label{lem:ptwisedoubling}
Let $(X,d,\mu)$ be a complete, pointwise doubling metric measure space. Then there exists a countable collection $\{A_i\}$ of closed subsets of $X$, along with constants $C_i>1$ and $R_i>0$, with the following properties:
\begin{enumerate}[(i)]
\item We have $\mu(X\setminus \cup_i A_i) = 0$.
\item Each $A_i$ is metrically doubling,
\item Each $A_i$ is $(C_i,R_i)$-uniformly pointwise doubling.
\end{enumerate}
\end{lemma}

\begin{rmk}\label{rmk:ptwiseblowup}
If $(X,d,\mu)$ is pointwise doubling (but not necessarily metrically doubling), then for $\mu$-a.e. $x\in X$, we can define the blowups of $X$ at $x$ as follows: Decompose $X$ into closed sets $A_i$ as in Lemma \ref{lem:ptwisedoubling} and, define the blowups of $X$ at $x\in A_i$ to be the blowups of $A_i$ at $x$, which are well-defined as $A_i$ is metrically doubling. For $\mu$-a.e. $x\in X$, this choice is independent of the choice of decomposition of $X$ (see \cite{BL15_RNP}, Section 9).

When we speak of the blowups of a pointwise doubling metric measure space $(X,d,\mu)$, as in Theorem \ref{thm:main}, this is what we mean. Observe that the blowups of a pointwise doubling space are metrically doubling metric spaces.

Note that if $(X,d,\mu)$ is metrically doubling and pointwise doubling, the blowups of $X$ in this sense may not coincide with the blowups of $X$ in the metric sense. For example, if $X$ is $\RR^2$ equipped with the measure $\mu$ which is the restriction of $\mathcal{H}^1$ to a single line, then $(X,d,\mu)$ is both metrically and pointwise doubling. However the blowups of $X$, in the sense of this remark, are lines almost everywhere.

If $\mu$ has the additional property that it assigns measure zero to porous subsets of $X$, then the two notions of blowup agree $\mu$-almost everywhere. (See Remark 7.2 in \cite{CKS15}.)
\end{rmk}

\subsection{Alberti representations}
We will not really need any properties of Alberti representations other than Proposition \ref{prop:LQblowup} below. However, for background we give the relevant definitions. For more on Alberti representations, we refer the reader to \cite{Ba15}, as well as \cite{Sc16}, \cite{CKS15}, and \cite{BL15}.

If $X$ is a metric space, let $\Gamma(X)$ denote the set of all bilipschitz functions
$$ \gamma: K \rightarrow X $$
where $K$ is a non-empty compact subset of $\RR$. We write $\text{Dom }\gamma$ for the domain $K$ of $\gamma$ and $\text{Im }\gamma$ for the image of $\gamma$ in $X$.

If $\gamma\in \Gamma(X)$, then the \textit{graph} of $\gamma$ is the compact set
$$ \{(t,x) \in \RR \times X : t\in K, \gamma(t)=x\}.$$
We endow $\Gamma(X)$ with the metric $d$ which sets $d(\gamma, \gamma')$ equal to the Hausdorff distance in $\RR\times X$ between the graphs of $\gamma$ and $\gamma'$. With this metric, $\Gamma(X)$ is a complete, separable metric space.

\begin{definition}
Let$(X,d,\mu)$ be a metric measure space, $\mathbb{P}$ a Borel probability measure on $\Gamma(X)$, and, for each $\gamma\in\Gamma(X)$, let $\nu_\gamma$ a Borel measure on $X$ that is absolutely continuous with respect to $\mathcal{H}^1|_{\text{Im }\gamma}$.

For a measurable set $A\subset X$, we say that $\mathcal{A} = (\mathbb{P}, \{\nu_\gamma\})$ is a \textit{Alberti representation of} $\mu|_A$ if, for each Borel set $Y\subset A$, 
\begin{itemize}
\item the map $\gamma \mapsto \nu_\gamma(Y)$ is Borel measurable, and
\item we have
$$ \mu(Y) = \int_{\Gamma(X)} \nu_\gamma(Y) d\mathbb{P}(\gamma).$$
\end{itemize}
\end{definition}

To specify the directions of Alberti representations, we define a cone in $\RR^n$ as folows: Given $w\in S^{n-1}$ and $\theta\in (0,1)$, let
$$ C(w,\theta) = \{ v\in \RR^n : v\cdot w \geq (1-\theta)\| v\|\}.$$
A collection $C_1, \dots, C_m$ of cones in $\RR^n$ are \textit{independent} if any choice of non-zero vectors $v_1\in C_1, \dots, v_m\in C_m$ form a linearly independent set.

Suppose $(X,d,\mu)$ is a metric measure space with Alberti representation $(\mathbb{P}, \{\nu_\gamma\})$. Let $\phi:X\rightarrow\RR^n$ be Lipschitz and let $C\subset \RR^n$ be a cone. We say that the Alberti representation $(\mathbb{P}, \{\nu_\gamma\})$ is in the \textit{$\phi$-direction of the cone $C$} if
$$ (\phi \circ \gamma)'(t) \in C \setminus \{0\} $$
for $\mathbb{P}$-a.e. $\gamma\in \Gamma(X)$ and a.e. $t\in \text{Dom }\gamma$.

Finally, if $\phi:X\rightarrow \RR^n$ is Lipschitz, we say that a collection $\mathcal{A}_1, \dots, \mathcal{A}_m$ of Alberti representations is \textit{$\phi$-independent} if there are independent cones $C_1, \dots, C_m$ in $\RR^n$ such that each $\mathcal{A}_i$ is in the $\phi$-direction of $C_i$.

\section{Lipschitz quotient mappings}\label{sec:LQ}
Lipschitz quotient mappings were first introduced in \cite{BJLPS99} in the context of Banach spaces.

\begin{definition}\label{def:LQ}
Let $X$ and $Y$ be metric spaces. A mapping $F:X\rightarrow Y$ is called a Lipschitz quotient (LQ) mapping if there is a constant $L\geq 1$ such that
\begin{equation}\label{eqn:LQ}
 B(F(x), r/L) \subseteq F(B(x,r)) \subseteq B(F(x),Lr)
\end{equation}
for all $x\in X$ and all $r>0$.

If we wish to emphasize the constant $L$, we will call such a map an $L$-LQ map.
\end{definition}
The second inclusion in \eqref{eqn:LQ} simply says that an $L$-LQ mapping is $L$-Lipschitz.

The way Lipschitz quotient mappings enter the proof of Theorem \ref{thm:main} is via the following result. It was proven (in slightly different language) in \cite{Sc13, Sc16}, Theorem 5.56 (see equation (5.96) in that paper) and (for doubling measures) in \cite{GCD15}, Corollary 5.1. A significantly stronger version of this result in the setting of Lipschitz differentiability spaces can be found in \cite{CKS15}, Theorem 1.11. All three of these results yield the following proposition with only minor changes.

\begin{prop}\label{prop:LQblowup}
Let $(X,d,\mu)$ be a metric measure space with $\mu$ pointwise doubling. Suppose that, for some Lipschitz function $\phi:X\rightarrow\RR^n$, $\mu$ has $n$ $\phi$-independent Alberti representations. Then for almost every $x\in X$, there is a constant $L\geq 1$ such that for every blowup $(\hat{X},\hat{x},\hat{\phi})$ of $(X,x,\phi)$, the mapping $\hat{\phi}$ is a Lipschitz quotient map of $\hat{X}$ onto $\RR^n$ with constant $L$.
\end{prop}

In the remainder of this section, we collect some basic properties of Lipschitz quotient mappings that will be used below.

The following path lifting lemma is one of the main tools used in \cite{BJLPS99} (Lemma 4.4) and \cite{JLPS00} (Lemma 2.2). We repeat it here in our context, along with its brief proof.

\begin{lemma}\label{lem:pathlifting}
Let $X$ be a proper metric space. Let $F:X\rightarrow Y$ be $L$-LQ, and let $\gamma:[0,T]\rightarrow Y$ be a $1$-Lipschitz curve with $\gamma(0)=F(x)$. Then there is a $L$-Lipschitz curve $\tilde{\gamma}:[0,T]\rightarrow X$ such that $\tilde{\gamma}(0)=x$ and $F\circ \tilde{\gamma} = \gamma$. 
\end{lemma}
\begin{proof}
Fix $m\in\mathbb{N}$. We define
$$\tilde{\gamma}_m: \left(\frac{1}{m}\mathbb{Z} \cap [0,T]\right) \rightarrow X$$
as follows.

Set $\tilde{\gamma}_m(0)= x$. By induction, assume that $\tilde{\gamma}_m(k/m)$ has been defined and $F(\tilde{\gamma}(k/m)) = \gamma(k/m)$.

We know that
$$ F(\overline{B}( \tilde{\gamma}_m(k/m), L/m)) \supseteq \overline{B}( \gamma(k/m), 1/m) \ni \gamma((k+1)/m). $$
Here the first inclusion follows from the fact that $F$ is a $L$-LQ mapping and that $X$ is proper, while the second inclusion follows from the fact that $\gamma$ is $1$-Lipschitz.

We therefore define  $\tilde{\gamma}_m(k/m)$ to be any point of $\overline{B}( \tilde{\gamma}_m(k/m), L/m)$ that maps onto $\gamma((k+1)/m)$ under $F$.

It follows that $\tilde{\gamma}_m$ is $L$-Lipschitz for each $m\in\mathbb{N}$. By a standard Arzel\`a-Ascoli type argument, a sub-sequence of $\{\tilde{\gamma}_m\}$converges as $m\rightarrow \infty$ to a curve $\tilde{\gamma}$ as desired.
\end{proof}

The following result is from Proposition 4.3 of \cite{BJLPS99}, or alternatively from the identical Proposition 2.1 of \cite{JLPS00}. These results are not stated in this form, but rather are stated only for mappings from $\RR^2$ to $\RR^2$. However the proof works exactly the same way in the more general setting below, using Lemma \ref{lem:pathlifting}.

\begin{prop}[\cite{BJLPS99} Proposition 4.3,  \cite{JLPS00} Proposition 2.1]\label{prop:LQdiscrete}
For each $L\geq 1$ and $D\geq 1$, there is a constant $N = N(L,D)$ with the following property:

Let $X$ be a proper, metrically $D$-doubling topological plane, and let $f:X\rightarrow \RR^2$ be a $L$-LQ mapping. Then $f$ is discrete, and furthermore
$$ \#f^{-1}(p) \leq N $$
for all $p\in\RR^2$. 
\end{prop}
\begin{proof}
A reading of the (identical) proofs of Proposition 4.3 of \cite{BJLPS99} and Proposition 2.1 of \cite{JLPS00} shows that the only requirement on the domain of the mapping is that it is a doubling topological plane.

With this remark in mind, what those proofs directly show is the following statement: If $X$ is a proper, metrically $D$-doubling topological plane and $f:X\rightarrow \RR^2$ is an $L$-LQ mapping, then $\#(f^{-1}(p) \cap B(x,1))$ is uniformly bounded (for all $p\in\RR^2$ and $x\in f^{-1}(p)\in X$) by a constant depending only on $L$ and $D$. To achieve the conclusion of Proposition \ref{prop:LQdiscrete}, one need only rescale and apply this result, for each $\lambda>0$, to the $L$-LQ mappings $x\mapsto \lambda f(x)$, considered as mappings on the metrically $D$-doubling topological plane $\lambda X$.
\end{proof}

We remark that, to our knowledge, it is open whether a result like Proposition \ref{prop:LQdiscrete} holds for Lipschitz quotient mappings from $\RR^n$ to $\RR^n$, for $n\geq 3$.

Proposition \ref{prop:LQdiscrete} will tell us that our blowup mappings are discrete open mappings between topological planes. The following result of \v{C}ernavskii-V\"ais\"ala is then relevant.

\begin{prop}[\cite{Va66}, Theorem 5.4]\label{prop:vaisala}
Let $f:M\rightarrow N$ be a continuous, discrete, and open mapping between topological $n$-manifolds. Then $f$ is a local homeomorphism off of a closed branch set $\mathcal{B}_f$ such that
$$ \dim \mathcal{B}_f \leq n-2 \text{ and } \dim f(\mathcal{B}_f)\leq n-2. $$
\end{prop}

\begin{rmk}
In the statement of Proposition \ref{prop:vaisala}, since $f$ is discrete and open and the branch set $\mathcal{B}_f$ is closed, the inequality $\dim f(\mathcal{B}_f)\leq n-2$ follows from \cite{CH60}, Lemma 2.1.
\end{rmk}

As Proposition \ref{prop:vaisala} will allow us to find locations where our blowup mappings are injective LQ mappings, we now analyze those locations further. Recall that a metric space is called \textit{geodesic} if every two points can be joined by a curve whose length is equal to the distance between the points.

\begin{lemma}\label{lem:injectivebilip}
Let $X$ be a proper metric space and let $Y$ be a geodesic metric space. Let $F:X\rightarrow Y$ be $L$-LQ, and suppose that $F$ is injective on $B(x,r)$ for some $x\in X, r>0$.

Then $F$ is $L$-bilipschitz on $B(x,r/(1+2L^2))$.
\end{lemma}
\begin{proof}
As an $L$-LQ mapping, $F$ is automatically $L$-Lipschitz. Consider distinct points $p,q\in B(x,r/(1+2L^2))$. Let $T=d(F(p),F(q))\leq \frac{Lr}{1+2L^2}$.

Let $\gamma:[0,T]\rightarrow Y$ parametrize by arc length a geodesic from $F(p)$ to $F(q)$. Then by Lemma \ref{lem:pathlifting} there is a $L$-Lipschitz curve $\tilde{\gamma}:[0,T]\rightarrow X$ such that $\tilde{\gamma}(0)=p$ and $F\circ \tilde{\gamma} = \gamma$. 

Since $\tilde{\gamma}$ is $L$-Lipschitz, we have that
$$ \diam(\tilde{\gamma}) \leq LT \leq \frac{L^2r}{1+2L^2}$$
and therefore
$$ d(\tilde{\gamma}(T),x) \leq \frac{r}{1+2L^2} + \frac{L^2r}{1+2L^2} < r.$$

Therefore $\tilde{\gamma}(T) \in B(x,r)$ and $F(\tilde{\gamma}(T))=\gamma(T)=F(q)$. Since $F$ is injective on $B(x,r)$, it follows that
$$\tilde{\gamma}(T) = q$$
and so
$$ d(p,q) = d(\tilde{\gamma}(0), \tilde{\gamma}(T))\leq  LT \leq L d(F(p),F(q)). $$
\end{proof}

\section{Convergence and closeness}\label{sec:close}
We introduced the notions of Gromov-Hausdorff convergence and blowups in Section \ref{sec:prelims}. Closely related to this type of convergence are the following notions of closeness between spaces.

\begin{definition}
Let $M$ and $N$ be metric spaces and fix $\eta>0$. We call a (not necessarily continuous) mapping $f:M\rightarrow N$ an \textit{$\eta$-isometry} if
$$ |d_N(f(x), f(y)) - d_M(x,y)| \leq \eta \text{ for all } x,y\in M. $$
\end{definition}
In other contexts, such mappings are also sometimes called  $(1,\eta)$-quasi-isometric embeddings or $(1,\eta)$-Hausdorff approximations.

\begin{definition}\label{def:close}
Let $(M,p)$ and $(N,q)$ be pointed metric spaces and let $t>0$, $\epsilon\in (0,1/10)$.  We will say that $(M,p)$ and $(N,q)$ are \textit{$\epsilon$-close at scale $t$} if there exist $\epsilon t $-isometries
\begin{equation}\label{eqn:close1}
f:B(p,t/\epsilon) \rightarrow N \text{ and } g:B(q,t/\epsilon) \rightarrow M
\end{equation}
such that $d(f(p),q) \leq \epsilon t$, $d(g(q),p) \leq \epsilon t$, and furthermore
\begin{equation}\label{eqn:close2}
d(f(g(y)),y) \leq \epsilon t \text{ and } d(g(f(x)), x) \leq \epsilon t
\end{equation}
for all $y\in B(q,t/2\epsilon)$ and $x\in B(p,t/2\epsilon)$.

If $\phi:M\rightarrow\RR^k$ and $\psi:N\rightarrow\RR^k$ are Lipschitz, we will say that the triples $(M,p,\phi)$ and $(N,q,\psi)$ are \textit{$\epsilon$-close at scale $t$} if the above holds and in addition
$$ |\phi \circ g - \psi| \leq \epsilon t, |\psi\circ f - \phi| \leq \epsilon t$$
everywhere on $B(q,t/\epsilon)$ and $B(p,t/\epsilon)$, respectively. 
\end{definition} 

\begin{rmk}\label{rmk:GHclose}
If $(X_n,p_n,\phi_n)$ is a sequence of triples converging in the pointed Gromov-Hausdorff sense to $(X,p,\phi)$, then for all $R,\epsilon>0$, there exists $N\in\mathbb{N}$ such that $(X_n, p_n, \phi_n)$ is $\epsilon$-close to $(X,p,\phi)$ at scale $R$, for all $n\geq N$. See, for example, Lemmas 8.11 and 8.19 of \cite{DS97} or Definition 8.1.1 of \cite{BBI01}
\end{rmk}

We collect some other simple observations about closeness below.
\begin{lemma}\label{lem:closeness}
Let $(M,p,\phi)$ and $(N,q,\psi)$ be $\epsilon$-close at scale $t$, with mappings $f$ and $g$ as in Definition \ref{def:close}. Fix $\lambda\in (0,1/2]$.
\begin{enumerate}[(a)]
\item \label{closeness1} If $B(x,\lambda t)\subseteq B(p,t)$, then $(M,x,\phi)$ and $(N,f(x),\psi)$ are $\frac{\epsilon}{\lambda}$-close at scale $\lambda t$.
\item \label{closeness2} If in addition $(N,q,\psi)$ and $(N',q',\psi')$ are $\delta$-close at scale $t$, then $(M,p,\phi)$ and $(N',q',\psi')$ are $2(\epsilon+\delta)$-close at scale $t$.
\end{enumerate}
\end{lemma}
\begin{proof}
For \eqref{closeness1}, we simply use the same mappings $f$ and $g$ that are provided by the fact that $(M,p,\phi)$ and $(N,q,\psi)$ are $\epsilon$-close at scale $t$. We now consider these as mappings between $(M,x,\phi)$ and $(N,f(x),\psi)$. The only aspect of Definition \ref{def:close} that is not immediately obvious is the fact that $f$ and $g$ are defined on $B(x,(\lambda t)/(\epsilon/\lambda))$ and $B(f(x),(\lambda t)/(\epsilon/\lambda))$, respectively. This follows from the triangle inequality and the assumptions that $\epsilon\leq 1/10$, $\lambda\leq 1/2$.

For \eqref{closeness2}, one may compose the relevant $\epsilon$-isometries, and check that the resulting maps satisfy Definition \ref{def:close} using the triangle inequality and the assumptions that $\epsilon, \delta \leq 1/10$.
\end{proof}

It is convenient for topological arguments to have a continuous version of closeness.

\begin{definition}
Let $M$ and $N$ be metric spaces and fix $\eta>0$. We say that continuous maps $f,g:M\rightarrow N$ between metric spaces are \textit{$\eta$-homotopic} if they are homotopic by a homotopy $H:M\times[0,1]\rightarrow N$ such that
$$ d_N (f(x), H(x,t)) \leq \eta $$
for all $x\in M$ and $t\in [0,1]$.
\end{definition}

\begin{definition}\label{def:ctsclose}
Let $(M,p)$ and $(N,q)$ be pointed metric spaces and let $t>0$, $\epsilon\in (0,1/2)$.  We will say that $(M,p)$ and $(N,q)$ are \textit{continuously $\epsilon$-close at scale $t$} if there exist continuous $\epsilon t $-isometries
$$ f:B(p,t/\epsilon) \rightarrow N \text{ and } g:B(q,t/\epsilon) \rightarrow M$$
such that
$$ d(f(p), q) \leq \epsilon t \text{ and } d(g(q),p)\leq \epsilon t, $$
and
$$ g\circ f |_{B(p,t/\epsilon)} $$
is $\epsilon t$-homotopic to the inclusion $B(p,t/\epsilon) \rightarrow M$, and similarly for $f\circ g$.

If $\phi:M\rightarrow\RR^k$ and $\psi:N\rightarrow\RR^k$ are Lipschitz, we will say that the triples $(M,p,\phi)$ and $(N,q,\psi)$ are \textit{continuously $\epsilon$-close at scale $t$} if the above holds and in addition
$$ |\phi \circ g - \psi| \leq \epsilon t, |\psi\circ f - \phi| \leq \epsilon t$$
where defined.
\end{definition}

The following result is useful for connecting closeness and continuous closeness.
\begin{lemma}\label{lem:semmes}
Fix $A, L\geq 1$. Then there is a constant $\Lambda=\Lambda(A,L)\geq 1$ with the following property.

Let $(M,p,\phi)$ and $(N,q,\psi)$ be triples such that $(M,p)$ and $(N,q)$ are pointed $A$-LC topological planes, and $\phi,\psi$ are $L$-Lipschitz. If $\epsilon\in (0,(10\Lambda)^{-1})$ and $t>0$, and if $(M,p,\phi)$ and $(N,q,\psi)$ are $\epsilon$-close at scale $t$, then they are continuously $\Lambda \epsilon$-close at scale $t$.
\end{lemma}
\begin{proof}
This follows from the general ``induction on skeleta'' arguments of \cite{Pe90} or \cite[Section 5]{Se96}. Here we give a direct proof based on these methods, while making no attempt to optimize constants. The idea is simply to triangulate $M$ and $N$ and use the linear contractibility to continuously extend the coarse mappings $f$ and $g$ from Definition \ref{def:close} successively from the vertices of the triangulation to the edges and then to the faces. To verify the homotopy inverse portion of Definition \ref{def:ctsclose}, one does a similar process on $M\times [0,1]$ and $N\times [0,1]$.

Let $f:B(p,t/\epsilon) \rightarrow N$ and $g:B(q,t/\epsilon) \rightarrow M$ be the mappings as in Definition \ref{def:close}.

Fix a triangulation $\mathcal{T}$ of $M$ such that each triangle $T\in\mathcal{T}$ has diameter at most $\epsilon t$. Let $\mathcal{T}_\epsilon$ be the collection of triangles in $\mathcal{T}$ that intersect $B(p,t/2\epsilon)$.

Let $f_0$ be the restriction of $f$ to the $0$-skeleton of $\mathcal{T}_\epsilon$. Note that if $x$ and $y$ are adjacent points in the $0$-skeleton of $\mathcal{T}_\epsilon$, then $d(f(x),f(y))\leq 3\epsilon t$.

The $A$-LC property of $N$ allows us to extend $f_0$ to a continuous map $f_1$ from the $1$-skeleton of $\mathcal{T}_\epsilon$ into $N$ with the property that
$$ \diam(f_1(\partial T)) \leq 9A\epsilon t \text{ for each } T\in\mathcal{T}_\epsilon.$$
A second application of the $A$-LC property of $N$ allows us to extend $f_1$ to a continuous map $f_2$ from
$$ \cup_{T\in\mathcal{T}_\epsilon} T \supset B(p,t/2\epsilon).$$
into $N$, with the property that
\begin{equation}\label{eqn:f2diam}
\diam(f_2(T)) \leq 18A^2\epsilon t \text{ for each } T\in\mathcal{T}_\epsilon.
\end{equation}

It follows from \eqref{eqn:f2diam} that the continuous map $f_2$ satisfies
\begin{equation}\label{eqn:ff2}
 d(f(x), f_2(x)) \leq 20A^2\epsilon t \text{ for all } x\in B(p, t/2\epsilon). 
\end{equation}

By the same method, we can find a continuous map $g_2:B(q,t/2\epsilon) \rightarrow M$ such that
\begin{equation}\label{eqn:gg2}
d(g(y), g_2(y)) \leq 20A^2\epsilon t \text{ for all } y\in B(q, t/2\epsilon). 
\end{equation}

It then follows from \eqref{eqn:ff2} and \eqref{eqn:gg2}, and the fact that $\phi$ and $\psi$ are $L$-Lipschitz, that
$$ |\phi \circ g_2 - \psi| \leq 21A^2L\epsilon t \text{ and } |\psi\circ f_2 - \phi| \leq 21A^2L\epsilon t$$
where defined.

We now argue that
$$ h = g_2\circ f_2 |_{B(p,t/2\epsilon)} $$
is $C\epsilon t$-homotopic to the inclusion $B(p,t/2\epsilon) \rightarrow M$, for some $C$ depending only on $A$. Of course, the same argument will show that $ \tilde{h} = f_2\circ g_2 |_{B(q,t/2\epsilon)} $ is $C\epsilon t$-homotopic to the inclusion $B(q,t/\epsilon) \rightarrow N$.

Observe that a simple triangle inequality calculation using \eqref{eqn:ff2} and \eqref{eqn:gg2} and the properties of $f$ and $g$ shows that
$$ d(h(x), x) \leq 42A^2\epsilon t \text{ for all } x\in B(p,t/2\epsilon).$$

Recall the previously defined triangulation $\mathcal{T}$ in $M$ and the collection $\mathcal{T}_\epsilon$ inside it. These yield triangulations of $M\times\{0\}$ and $M\times\{1\}$. We can then obtain a triangulation $\mathcal{S}$ of $M\times [0,1]$ with no additional vertices by simply triangulating each product $T\times [0,1]$ for triangles $T\in \mathcal{T}$. Note that, under this construction, if $S$ is a simplex of $\mathcal{S}$ that intersects $B(p,t/2\epsilon) \times [0,1]$, then
$$ S \cap (M\times \{0,1\}) \subseteq \left(\bigcup_{T\in\mathcal{T}_\epsilon} T \right)\times\{0,1\}.$$

Let $\mathcal{S}_\epsilon$ be the collection of simplices in $\mathcal{S}$ that intersect $B(p,t/2\epsilon) \times [0,1]$.

We define a map $H_1$ from the $1$-skeleton $\mathcal{S}^1_\epsilon$ of $\mathcal{S}_\epsilon$ to $M$ as follows: On edges of $\mathcal{S}_\epsilon$ in $M\times \{0\}$, $H_1$ agrees with the identity.  On edges of $\mathcal{S}_\epsilon$ in $M\times \{1\}$, $H_1$ agrees with $h$. On each remaining edge, we extend $H_1$ continuously from its values at the endpoints, using the $A$-LC property of $M$. Then for each edge $e$ of $\mathcal{S}_\epsilon^1$ between points $x\times a$ and $y\times b$ in $\mathcal{S}_\epsilon^0$ ($x,y\in M$, $a,b\in \{0,1\}$), we have
$$ \diam(H_1(e)) \leq 2A d(H_1(x),H_1(y)) \leq 84A^3 \epsilon t. $$
We have now defined  $H_1$ on the $1$-skeleton of $\mathcal{S}_\epsilon$. We now extend to a map $H_2$ on the $2$-skeleton of $\mathcal{S}_\epsilon$. For each face of $\mathcal{S}_\epsilon$ in $M\times\{0\}$, define $H_2$ by the identity and for each face of $\mathcal{S}_\epsilon$ in $M\times \{1\}$,define $H_2$ by $h$; note that this continuously extends $H_1$. On each remaining face of $\mathcal{S}_\epsilon$, we define $H_2$ as an extension of $H_1$ using the $A$-LC property of $M$. Since the image of each edge of $\mathcal{S}_\epsilon$ under $H_1$ has diameter at most $84A^3 \epsilon t$, the image of each face of $\mathcal{S}_\epsilon$ under $H_2$ has diameter at most $336A^4 \epsilon t$.

Finally, we extend $H_2$ to a map $H$ on the union of simplices of $\mathcal{S}_\epsilon$, again using the $A$-LC property of $M$. The image of each simplex under $H$ has diameter at most $1344A^5\epsilon t$.

Since the union of simplices of $\mathcal{S}_\epsilon$ contains all of $B(p,t/2\epsilon)\times [0,1]$, a restriction of $H$ is a homotopy between $h$ and the inclusion of $B(p,t/2\epsilon)$ into $M$. 

In addition, since $\diam(H(S))\leq 1344A^5\epsilon t$ for each simplex $S$ of $\mathcal{S}_\epsilon$, it follows that 
$$ d(H(x,t), x) \leq 1344A^5\epsilon t$$
for all $(x,t)\in B(p,t/2\epsilon)\times [0,1]$.

This proves that $h= g_2\circ f_2 |_{B(p,t/2\epsilon)}$ is $1344A^5\epsilon t$-homotopic to the inclusion map of $B(p,t/2\epsilon)$ into $M$. The same argument proves the analogous statement for $f_2\circ g_2|_{B(q,t/2\epsilon)}$

Therefore, choosing $\Lambda \geq \max( 1344A^5, 21A^2L)$ completes the proof of the lemma.
\end{proof}

\section{Topological lemmas}\label{sec:topology}
We first make the following simple observation about $A$-LC spaces.
\begin{lemma}\label{lem:balldomain}
Let $X$ be a proper $A$-LC space. Fix $x\in X$ and $r>0$. Then there is a connected open set $U$ and a connected compact set $K$ such that
\begin{equation}\label{eqn:openballdomain}
B(x,r) \subseteq U \subseteq B(x, 2Ar)
\end{equation}
and
\begin{equation}\label{eqn:closedballdomain}
B(x,r) \subseteq K \subseteq B(x, 2Ar).
\end{equation}
\end{lemma}
\begin{proof}
The existence of $U$ satisfying \eqref{eqn:openballdomain} is shown in Lemma 2.11 of \cite{GCD16}.

To find a continuum $K$ as in \eqref{eqn:closedballdomain}, consider the homotopy $H$ which contracts $B(x,2r)$ in $B(x, 2Ar)$ and set
$$ K = H(\overline{B}(x,r) \times [0,1]).$$
\end{proof}

Now let $X$ and $Y$ be homeomorphic to $\RR^n$ and let $f:X\rightarrow Y$ be a proper continuous mapping. (Here \textit{proper} means that $f^{-1}(K)$ is compact in $X$ whenever $K$ is compact in $Y$.) Then $f$ extends naturally to a continuous mapping between the one-point compactifications of $X$ and $Y$, which are homeomorphic to $S^n$.

For a domain $D\subset X$, and a point $y\in  Y\setminus f(\partial D)$, we can therefore use $\mu(y,D,f)$ to denote the local degree of $f$, as defined on p. 16 of \cite{Ri93}. That is, if $f$ is a continuous map from a domain $D\subset X \subset S^n$ into $Y \subset S^n$ and $y\notin f(\partial D)$, then $\mu(y,D,f)$ is defined by considering the following sequence of induced mappings on singular homology of pairs:
\begin{equation*}
\begin{tikzcd}[column sep = small]
H_n(S^n) \arrow{r}{j{*}} & H_n(S^n, S^n \setminus (D\cap f^{-1}(y))) & H_n(D, D\setminus f^{-1}(y)) \arrow{l}{e_{*}} \arrow{r}{f_{*}} & H_n(S^n, S^n\setminus \{y\})  & H_n(S^n) \arrow{l}{k_{*}}
\end{tikzcd}
\end{equation*}

Here $j$, $e$, and $k$ are inclusions. The homomorphism $e_*$ is an isomorphism by excision, and $k_*$ is an isomorphism because $S^n\setminus\{y\}$ is homologically trivial. There is an integer $\mu$ such that the homomorphism $k_*^{-1} f_* e_*^{-1} j_*$ sends each $\alpha\in H_n(S^n)$ to a multiple $\mu \alpha \in H_n(S^n)$. This integer $\mu$ is the local degree $\mu(y,D,f)$.

The following basic properties of the local degree can be found in Proposition 4.4 of \cite{Ri93}.
\begin{enumerate}[(a)]
\item\label{item:degree1} The function $y\rightarrow \mu(y,D,f)$ is constant on each connected component of $Y\setminus f(\partial D)$. 

\item\label{item:degree2} If $f : D \rightarrow f(D)$ is a homeomorphism, then $\mu(y, D, f ) = \pm 1$ for each $y \in f (D)$.

\item\label{item:degree3} If $y\in Y \setminus f(\partial D)$ and $f^{-1}(y) \subset D_1 \cup \dots \cup D_p$, where $D_i$ are disjoint domains in $D$ such that $y\in Y \setminus \partial D_i$ for each $i$, then
$$ \mu(y,D,f) = \sum_{i=1}^p \mu(y, D_i, f).$$

\item\label{item:degree4} If $f$ and $g$ are homotopic by a homotopy $H_t$, $t\in [0,1]$, and if $y\notin H_t(\partial D)$ for all $t\in [0,1]$, then $\mu(y,f,D) = \mu(y,g,D)$.
\end{enumerate}

For Lipschitz quotient maps, having local degree $\pm 1$ is enough to guarantee local injectivity in the following sense.
\begin{lemma}\label{lem:degree1impliesinjective}
Let $Z$ be metrically doubling and homeomorphic to $\RR^2$, and let $f:Z\rightarrow\RR^2$ be a LQ mapping. Let $D$ be a domain in $Z$. Let $y$ be a point in $\mathbb{R}^n \setminus f(\partial D)$ such that $|\mu(y,D,f)|= 1$.

Let $U$ be the connected component of $\mathbb{R}^n \setminus f(\partial D)$ containing $y$, and let $D'$ be a connected component of  $f^{-1}(U)$ in $D$. Assume that $f^{-1}(y) \cap D \subset D'$.

Then $f$ is injective on $D'$. 
\end{lemma}
\begin{proof}
As $f:Z\rightarrow \RR^2$ is a LQ mapping, it is continuous, open, and discrete (by Proposition \ref{prop:LQdiscrete}). Furthermore, the Lipschitz quotient property and fact that $\# f^{-1}(p)$ is finite for each $p$ (by Proposition \ref{prop:LQdiscrete}) implies that $f$ is a proper mapping.

As noted in Remark 3.2 of \cite{HR02}, or on p.\ 18 of \cite{Ri93}, Proposition \ref{prop:vaisala} implies that $f:Z\rightarrow \RR^n$ is either sense-preserving or sense-reversing. (This is because the branch set $\mathcal{B}_f$, having topological dimension at most $n-2$, cannot separate $D$.) Without loss of generality, we may assume that $f$ is sense-preserving, i.e. that
$$ \mu(p, V, f) > 0 $$
for every pre-compact domain $V\subseteq X$ and every point $p\in f(V) \setminus f(\partial V)$. 

Suppose that a point $z\in U$ has pre-images $x_1, \dots, x_k$ in $D'$. For each $i=1, 2, \dots, k$, place small disjoint domains $D_i$ in $D'$ such that $x_i\in D_i$ and $\partial D_i$ avoids the finite set $f^{-1}(z)$. 

Then 
$$ 1 = \mu(y,D,f) = \mu(y,D',f) = \mu(z,D',f) = \sum_{i=1}^k \mu(z,D_i,f). $$
Here the first equation is by assumption, the second and fourth are from property \eqref{item:degree3} of local degree, and the third follows from property \eqref{item:degree1}, since $y$ and $z$ are in the same connected component of $\RR^n \setminus f(\partial D')$.

Since $\mu(z,D_i,f)>0$ for each $i$, we must have that $k=1$. This shows that $f|_{D'}$ is injective.
\end{proof}

\begin{lemma}\label{lem:GHdegree1}
Let $X$ and $Y$ be $A$-LC and homeomorphic to $\RR^n$. Let $t>0$ and $0<\epsilon<1/100A^2$. 

Suppose that $(X,p_X)$ and $(Y,p_Y)$ are continuously $\epsilon$-close at scale $t$, with mappings $f,g$ as in Definition \ref{def:ctsclose}. Let $D_X$ and $D_Y$ be domains in $X$ and $Y$, respectively, and fix $t'>\epsilon t$.

Suppose also that $K_X \subset D_X$ and $K_Y \subset D_Y$ are compact connected sets such that
\begin{equation}\label{eqn:KXdef}
B(p_X,10At') \subset K_X \subset B(p_X, 5t) \subset D_X \subset B(p_X, 10At) 
\end{equation}
and
\begin{equation}\label{eqn:KYdef}
B(p_Y, t') \subset K_Y \subset B(p_Y, 2At') \subset B(p_Y, 11At) \subset D_Y \subset B(p_X, 22A^2t).
\end{equation}

Then the induced mapping
$$ f_{*} : H_n(D_X, D_X\setminus K_X) \rightarrow H_n(D_Y, D_Y \setminus K_Y)$$
is surjective.
\end{lemma}
\begin{proof}

The mapping $g\circ f|_{D_X}$ is $\epsilon t$-homotopic to the inclusion of $D_X$ into $X$. We consider $X$ as embedded in its one-point compactifiaction $S^n$. Hence, the mapping
$$ (g|_{D_Y})_* \circ (f|_{D_X})_* : H_n(D_X, D_X\setminus K_X) \rightarrow H_n(D_Y, D_Y\setminus K_Y) \rightarrow H_n(S^n, S^n\setminus \{p_X\}) $$
is the same map as the one induced by inclusion. That map is an isomorphism, by excision and duality (\cite{Sp81}, 4.6.5 and 6.2.17). 

Since all the groups are isomorphic to $\mathbb{Z}$, the first map must be surjective.
\end{proof}

\begin{lemma}\label{lem:degree1}

Fix $A,L\geq 1$. Let $X$ and $Y$ be $A$-LC topological planes, and let $\phi_X:X\rightarrow \RR^2$, $\phi_Y:Y\rightarrow\RR^2$ be $L$-LQ mappings. Let $t>0$ and $0<\epsilon< 1/100A^3L^3$. 

Suppose that $(X,p_X, \phi_X)$ and $(Y,p_Y, \phi_Y)$ are continuously $\epsilon$-close at scale $t$. In addition, suppose that $\phi_Y$ is a $L$-bilipschitz on $B(p_Y, 22A^2 t)$.

Then $\phi_X$ is $L$-bilipschitz on $B\left(p_X, \frac{t}{2AL^2(1+2L^2)}\right)$.
\end{lemma}
\begin{proof}
Let $y=\phi_X(p_X)$ and let $D_Y$ be a domain that contains $B(p_Y, 11A t)$ and is contained in $B(p_Y, 22A^2t)\subset B(p_Y, t/\epsilon)$. This exists by Lemma \ref{lem:balldomain} and our choice of $\epsilon$.

Because $\phi_Y$ is $L$-LQ and $\epsilon <1/100AL$, we see that
$$\phi_Y(D_Y) \supseteq B(\phi_Y(p_Y),10t/L)\supset B(\phi_Y(p_Y), 10L\epsilon t)\ni y.$$
Let $z=\phi_Y^{-1}(y) \cap D_Y$, which is a single point because $\phi_Y$ is bilipschitz on $D_Y$. 

Let $D_X$ be a domain containing $B(p_X,5t)$ and contained in $B(p_X, 10At)\subset B(p_X, t/\epsilon)$.

Let $t'=5L^2\epsilon t$. The following facts follow easily from our assumptions and the triangle inequality:
$$ z\in B(p_Y, t') \text{ and } f^{-1}(z) \subset B(p_X, t').$$

Hence (using Lemma \ref{lem:balldomain}) we can find compact connected sets $K_X$ and $K_Y$ such that
$$ f^{-1}(z) \subset B(p_X, 10At') \subset K_X \subset D_X$$
and
$$ B(p_Y, t') \subset K_Y \subset B(p_Y, 2At') \subset D_Y.$$
Observe that $z\in K_Y$, so $y\in \phi_Y(K_Y)$.

Now consider the commutative diagram below.
\begin{equation*}
\begin{tikzcd}[row sep = large, column sep = tiny]
H_2(S^2) \arrow{d}{\overline{j}_{*}}\\
H_2(S^2, S^2\setminus K_X) \arrow{d}{(i_1)_{*}}& H_2(D_X, D_X\setminus K_X) \arrow{l}{\overline{e}_{*}} \arrow{r}{\overline{f}_{*}} \arrow{d}{(i_2)_{*}} & H_2(D_Y, D_Y\setminus K_Y) \arrow{r}{\overline{(\phi_Y)}_{*}}\arrow{d}{(i_3)_{*}} & H_2(S^2, S^2\setminus \phi_Y(K_Y))\arrow{d}{(i_4)_{*}} \\
H_2(S^2, S^2 \setminus  f^{-1}(z)) & H_2(D_X, D_X\setminus f^{-1}(z)) \arrow{l}{e_{*}} \arrow{r}{f_{*}} & H_2(D_Y, D_Y\setminus \{z\}) \arrow{r}{(\phi_Y)_{*}} & H_2(S^2, S^2\setminus \{y\})\\
& & & H_2(S^2) \arrow{u}{k_{*}}
\end{tikzcd}
\end{equation*}

In this diagram, the homomorphisms are all induced by inclusion, except those labeled $f_*$ and $\overline{f}_*$, which are induced by $f$, and those labeled $(\phi_Y)_*$ and $(\overline{\phi_Y})_*$, which are induced by $\phi_Y$. The homomorphisms $e_*$, $\overline{e}_*$, and $k_*$ are isomorphisms, as in the definition of local degree. The homomorphisms $\overline{j}_*$ and $(i_4)_*$ are surjective, by duality (\cite{Sp81}, 6.2.17).

Following this diagram from top left to bottom right along the third row gives the local degree $\mu(\phi_Y \circ f, y, D_X)$. Following from top left to bottom right along the second row shows that the overall map is surjective. (We use Lemma \ref{lem:GHdegree1} for $\overline{f}_*$ in the second row.)  Hence
$$|\mu( y, D_X, \phi_Y \circ f)| = 1.$$

Now, we know that $\sup_{\overline{D_X}}|\phi_Y\circ f - \phi_X| \leq \epsilon t$. It follows that
\begin{equation}\label{eqn:farboundary}
\dist(y,\phi_X(\partial D_X))\geq \frac{1}{L} t > 10\epsilon t.
\end{equation}

Indeed, if $q\in\partial D_X$, then $d(q,p_X)\geq 5t$, so $d(f(q), f(p_X))\geq 5t-\epsilon t$, and so
$$ d(y,\phi_X(q)) = d(\phi_X(p_X), \phi_X(q)) \geq d(\phi_Y(f(p_Y)), \phi_Y(f(q))) - 2\epsilon t \geq L^{-1}(5-\epsilon)t - 2\epsilon t \geq \frac{1}{L} t,$$
which proves \eqref{eqn:farboundary}.

Hence, the homotopy invariance of local degree implies that
$$|\mu(y, D_X,\phi_X)| = 1.$$

We now aim to apply Lemma \ref{lem:degree1impliesinjective}. Let $U$ be the connected component of $\RR^2 \setminus \phi_X(\partial D_X)$ containing $y$, and let $D'_X$ be the connected component of $D_X \cap \phi_X^{-1}(U)$ containing $p_X$. Equation \eqref{eqn:farboundary} implies that $U$ contains $B(y, \frac{1}{L} t)$. Hence, since $\phi_X$ is $L$-Lipschitz, $D_X \cap \phi_X^{-1}(U)$ contains $B(p_X, t/L^2)$. It then follows from Lemma \ref{lem:balldomain} that the connected component $D'_X$ contains $B(p_X, t/2AL^2)$.

Another simple argument shows that
$$ \phi_X^{-1}(y) \cap D_X \subset B(p_X, 10L^2\epsilon t) \subset B(p_X, t/2AL^2) \subset D'_X. $$

An application of Lemma \ref{lem:degree1impliesinjective} shows that $\phi_X$ is injective on $D'_X$, hence on $B(p_X, t/2AL^2)$. Applying Lemma \ref{lem:injectivebilip} shows that $\phi_X$ is therefore $L$-bilipschitz on $B(p_X, t/2AL^2(1+2L^2))$.
\end{proof}

\section{Compactness results}\label{sec:compactness}

We will make use of the following completeness property of $A$-LC topological planes and $L$-LQ mappings
\begin{lemma}\label{lem:LQlimit}
Fix constants $L,D,A\geq 1$. Let
$$ \{ (X_n, p_n, f_n \colon X\rightarrow \RR^2) \} \rightarrow (X,p,f \colon X\rightarrow \RR^2)\} $$
be a sequence converging in the pointed Gromov-Hausdorff sense. Suppose that each $X_n$ is an $A$-LC, metrically $D$-doubling, topological plane, and that each $f_n$ is $L$-LQ. Then $X$ is a metrically $D$-doubling topological plane and $f$ is an $L$-LQ mapping.

Furthermore, $X$ is $A'$-LC for some $A'$ depending only on $A$ and $D$, and if each $X_n$ is $\lambda$-ALC for some fixed $\lambda\geq 1$, then $X$ is $\lambda$-ALC.
\end{lemma}
\begin{proof}
The following parts of the lemma are standard and simple to prove from the definitions: $X$ is $D$-doubling, $f$ is $L$-LQ, and if $X_n$ are all $\lambda$-ALC then $X$ is $\lambda$-ALC.

That $X$ is $A'$-LC for some $A'$ depending only on $A$ and $D$ appears in Lemma 2.12 of  \cite{GCD16}.

It remains only to show that $X$ is a topological plane. This follows from, e.g., Proposition 2.19 of \cite{GCD16} (note that Ahlfors regularity is not really required in that result, only metric doubling). Indeed, that result shows that $X$ is a linearly contractible (hence simply connected) homology $2$-manifold. We then note that all homology $2$-manifolds are topological $2$-manifolds (\cite{Br97}, Theorem V.16.32), and that the plane is the only simply connected non-compact $2$-manifold.
\end{proof}

We now use some compactness arguments to show that Lipschitz quotient mappings are quantitatively bilipschitz on balls of definite size.

\begin{prop}\label{prop:subball}
For each $L,D,A\geq 1$, there is a constant $s_0 = s_0(L,D,A)$ with the following property:

Let $X$ be a metrically $D$-doubling, $A$-LC, topological plane, and let $f \colon X\rightarrow\RR^2$ be a $L$-LQ mapping. Then in each ball $B(p,r)$ in $X$, there is a ball $B(q,s_0 r)\subseteq B(p,r)$ such that $f$ is $L$-bilipschitz on $B(q,s_0 r)$.
\end{prop}

The proof requires the following preliminary lemma. 

\begin{lemma}\label{lem:limitinglemma}
Let $\{(X_n, p_n, f_n\colon X_n\rightarrow\mathbb{R}^2)\}$ be a sequence with each $X_n$ metrically $D$-doubling, $A$-LC, and homeomorphic to $\RR^2$ and each $f_n$ $L$-LQ. Suppose that this sequence converges in the pointed Gromov-Hausdorff sense to a triple $(X,p,f \colon X \rightarrow \RR^2)$. Then there exists $s_0>0$ such that for all $n$ sufficiently large, $f_n$ is $L$-bilipschitz on a ball of radius $s_0$ in $B_n(p_n, 1)$. 
\end{lemma}
\begin{proof}
First of all, we observe by Lemma \ref{lem:LQlimit} that the space $X$ is an $A'$-LC, metrically $D$ doubling, topological plane, (with $A'=A'(A,D)\geq A$) and that the limit function $f$ is $L$-LQ. Hence, $f$ is discrete by Proposition \ref{prop:LQdiscrete}. By Proposition \ref{prop:vaisala}, $f$ is a homeomorphism on some sub-ball in $B(x,1)$, and so by Lemma \ref{lem:injectivebilip}, $f$ is $L$-bilipschitz on some ball $B(p',s)\subset B_X(p,1)$. 

We now claim that, for some $s_0>0$ and all $n$ sufficiently large, the mapping $f_n$ is bilipschitz on a ball $B(p'_n, s_0)\subset B_n(p_n,1)$. This will yield a contradiction.

Fix $\delta =  (2000A'L)^{-4} s$. If $n$ is sufficiently large, then $(X_n, p_n, f_n)$ and $(X,p,f)$ are $\delta$-close at scale $1$. Hence, by Lemma \ref{lem:closeness}, there are points $p'_n\in X_n$ such that the triples $(X_n, p'_n, f_n)$ and $(X,p', f)$ are $\frac{20A'\delta}{s}$-close at scale $t=s/20A'$.

Fix $n$ large as above. Then $(X_n, p'_n, f_n)$ and $(X,p',f)$ are $\epsilon$-close at scale $t$ (where $\epsilon =\frac{20A'\delta}{s}<1/100A'L^4$), and $f$ is $L$-bilipschitz on $B(p',s)\supseteq B(p', 20A't)$.

Thus, by Lemma \ref{lem:degree1}, $f_n$ is $L$-bilipschitz on $B(p'_n, t/2A'L^2(1+2L^2)) = B(p'_n,s_0)$. This completes the proof.
\end{proof}

\begin{proof}[Proof of Proposition \ref{prop:subball}]
We argue by contradiction. If the Proposition fails, then for some constants $D$, $A$, $L$, there is a sequence
$$ \left\{ (X_n, p_n, f_n) \right\} $$
of metrically $D$-doubling, $A$-LC, $L$-LQ topological planes and radii $r_n>0$ such that $f_n$ fails to be $L$-bilipschitz on each ball $B(p'_n, r_n/n)\subset B(p_n, r_n)$.

Consider the sequence
\begin{equation}\label{rescaledseq}
\left\{ (r_n^{-1} X_n, p_n, g_n) \right\},
\end{equation}
where $g_n(x) = r_n^{-1} f_n(x)$.

The spaces $r_n^{-1}X_n$ are still metrically $D$-doubling, $A$-LC topological planes, and the mappings $g_n:r_n^{-1}X_n\rightarrow \RR^2$ are $L$-LQ. Furthermore, the map $g_n$ fails to be $L$-bilipschitz on each ball $B(p'_n, 1/n)\subset B(p_n, 1)\subset \frac{1}{r_n} X_n$.

Consider a convergent subsequence of the sequence in \eqref{rescaledseq}. By Lemma \ref{lem:limitinglemma}, we see that there is a constant $s_0>0$ such that, for arbitrarily large values of $n\in\mathbb{N}$, $g_n$ is $L$-bilipschitz on a ball of radius $s_0$ in $B(p_n, 1)$. This is a contradiction.
\end{proof}

\section{Proof of Theorem \ref{thm:main} and Corollary \ref{cor:main}}\label{sec:proof}

\begin{proof}[Proof of Theorem \ref{thm:main}]
Fix a metric measure space $(X,d,\mu)$, a Borel set $U\subset X$, and a Lipschitz function $\phi:X\rightarrow \RR^n$ as in the assumptions of Theorem \ref{thm:main}. When proving Theorem \ref{thm:main}, we may assume without loss of generality that $U$ is closed, metrically doubling, and uniformly pointwise doubling. This assumption is justified by Lemma \ref{lem:ptwisedoubling}. By the definition of blowups for pointwise doubling spaces (see Remark \ref{rmk:ptwiseblowup}), this means that, for a.e. $x\in U$, blowups of $X$ at $x$ are the same as blowups of $U$ at $x$.

We first prove that $n\leq 2$, the first half of Theorem \ref{thm:main}. By Proposition \ref{prop:LQblowup}, there is a point $x\in U$ and a blowup $(\hat{X}, \hat{x}, \hat{\phi})$ of $(X,x,\phi)$ which is a metrically doubling topological plane and such that $\hat{\phi}$ is a Lipschitz quotient map onto $\RR^n$. Suppose $n> 2$. If $\pi:\RR^n\rightarrow\RR^2$ is the projection onto the first two coordinates, then $\pi\circ \hat{\phi}$ is a Lipschitz quotient mapping from $\hat{X}$ onto $\RR^2$. By Propositions \ref{prop:LQdiscrete} and \ref{prop:vaisala}, $\pi\circ\hat{\phi}$ is a homeomorphism on some ball $B(\hat{x}, t)\subset \hat{X}$. It is therefore impossible for $\hat{\phi}(B(\hat{x},t))$ to contain a ball in $\RR^n$, which contradicts the fact that $\hat{\phi}$ is a Lipschitz quotient map onto $\RR^n$. Therefore, $n\leq 2$.

We now assume $n=2$ and proceed to show that in this case $\mu|_U$ is $2$-rectifiable. This will complete the proof of Theorem \ref{thm:main}.

By assumption, at $\mu$-a.e. point $x\in U$, each blowup of $X$ is a topological plane. By Proposition \ref{prop:blowupsALC} in the appendix, it follows that for $\mu$-a.e. $x\in U$, there is a constant $\lambda(x)$ such that each blowup of $X$ at $x$ is $\lambda(x)$-ALC.

Thus, after our reductions, we have that, for $\mu$-a.e. $x\in U$, each blowup $(\hat{X}, \hat{x}, \hat{\phi})$ of $(X,x,\phi)$ at $x$ has the following properties:
\begin{enumerate}[(i)]
\item\label{item:blowup1} $\hat{X}$ is a metrically $D$-doubling, $\lambda(x)$-ALC topological plane.
\item\label{item:blowup2} $\hat{\phi}:\hat{X}\rightarrow\RR^2$ is a Lipschitz quotient map.
\end{enumerate}

Fix constants $L,\lambda\geq 1$. Let $A=A(\lambda)$ as provided by Lemma \ref{lem:ALCLC}. Let $\epsilon = \epsilon(L,D,A) = \left(\frac{s_0}{1000AL\Lambda(1+2L^2)}\right)^{5}$, where $s_0 = s_0(L,D,A)$ is as in Proposition \ref{prop:subball} and $\Lambda = \Lambda(A,L)$ is as in Lemma \ref{lem:semmes}. Fix $r_0>0$ and define $Y=Y_{L,\lambda, r_0}\subset U$ by
\begin{align}\label{eqn:Ydef}
Y =  \{ x\in U &: \text{ for all } r\in (0,r_0) , \text{ there is an }\\
&\lambda\text{-ALC}, D\text{-doubling}, L\text{-LQ plane }  (W, w, \psi) \text{ such that } \nonumber\\
&(r^{-1}U,x,r^{-1}(\phi-\phi(x))) \text{ and } (W, w, \psi) \text{ are } \epsilon\text{-close at scale } 1\} \nonumber
\end{align}

For each fixed $L, \lambda \geq 1$ and $r_0>0$, it follows from Lemma \ref{lem:LQlimit} that the above set $Y$ is closed in $U$, hence in $X$.

We can write $U$, up to an exceptional subset of measure zero, as a countable union of sets $Y$ as above, by varying $L, \lambda \in\mathbb{N}$ and  $r_0 \in \{1, 1/2, 1/3, \dots\}$. This follows from the rephrasing of Gromov-Hausdorff convergence in terms of closeness given in Remark \ref{rmk:GHclose}.

To prove Theorem \ref{thm:main}, it therefore suffices to show that $\mu|_Y$ is $2$-rectifiable. 

We first make the following simple rescaling observation.

\begin{lemma}\label{lem:rescaling}
For each $x\in Y$ and $0<r<r_0$, there is an $A$-LC, $L$-LQ plane $(\hat{X}, \hat{x}, \hat{\phi})$ such that $(U,x,\phi)$ and $(\hat{X}, \hat{x}, \hat{\phi})$ are $\epsilon$-close at scale $r$.
\end{lemma}
\begin{proof}
Given $x\in Y$ and $r<r_0$, let $(W,w,\psi)$ be as provided by the definition of $Y$. Note that $W$ is $\lambda$-ALC and hence $A$-LC by our choice of $A=A(\lambda)$ from Lemma \ref{lem:ALCLC}. The rescaled and translated triple
$$ (\hat{X}, \hat{x}, \hat{\phi}) := (rW, w, r\psi + \phi(x))$$
is then the desired one.
\end{proof}

\begin{lemma}\label{lem:iteration}
Let $q\in Y$ be a point and let $r\in (0,r_0)$. Let $(\hat{X}, \hat{q}, \hat{\phi})$ be an $A$-LC, $L$-LQ topological plane such that $(U,q,\phi)$ and $(\hat{X}, \hat{q}, \hat{\phi})$ are $\frac{10}{s_0}\epsilon$-close at scale $r$.

Suppose further that $\hat{\phi}$ is $L$-bilipschitz on $B(\hat{q},r)$.

Then $\phi$ is $2L$-bilipschitz on $B(q,r/20) \cap Y$.
\end{lemma}
\begin{proof}
Let $s_1 = \frac{1}{100A^2L^2(1+2L^2)}$ and let $\epsilon' = 100s_0^{-1}AL^2(1+2L^2)\epsilon \in (0,1/10)$. 

We will first prove the following claim.

\begin{claim}\label{claim:iteration}
For any $x\in B(q,r/20) \cap Y$ and $k\geq 0$, there is an $A$-LC, $L$-LQ plane $(\tilde{X}, \tilde{x}, \tilde{\phi})$ such that $(U,x,\phi)$ is $\epsilon'$-close to  $(\tilde{X}, \tilde{x}, \tilde{\phi})$ at scale $s_1^k r/10$ and such that $\tilde{\phi}$ is $L$-bilipschitz on $B(\tilde{x}, s_1^kr/5)$. 
\end{claim}
\begin{proof}[Proof of Claim \ref{claim:iteration}]
The proof is by induction on $k\geq 0$.

If $k=0$, we set $(\tilde{X}, \tilde{x}, \tilde{\phi})$ to be the triple $(\hat{X}, \hat{x}, \hat{\phi})$, where $\hat{x}\in B(\hat{q},r/10)\subset \hat{X}$ is chosen so that $(U,x,\phi)$ and $(\hat{X}, \hat{x}, \hat{\phi})$ are $\frac{100}{s_0}\epsilon$ close at scale $r/10$. (Here we use Lemma \ref{lem:closeness}.) Since $\hat{\phi}$ is $L$-bilipschitz on $B(\hat{x}, r/5)\subset B(\hat{q},r)$, and since $\frac{100}{s_0}\epsilon < \epsilon'$, we have proven the claim if $k=0$.

Now suppose $k>0$. Let $(\tilde{X},\tilde{x},\tilde{\phi})$ be an $A$-LC, $L$-LQ triple such that $(U,x,\phi)$ and $(\tilde{X},\tilde{x},\tilde{\phi})$ are $\epsilon$-close at scale $4AL^2(1+2L^2)s_1^k r/10$. (Such a triple is provided by Lemma \ref{lem:rescaling} and the fact that $x\in Y$.) 

By induction, we also have an $A$-LC, $L$-LQ plane $(\tilde{Z}, \tilde{z}, \tilde{\psi})$ such that $(U,x,\phi)$ is $\epsilon'$-close to  $(\tilde{Z}, \tilde{z}, \tilde{\psi})$ at scale $s_1^{k-1} r/10$ and such that $\tilde{\psi}$ is $L$-bilipschitz on $B(\tilde{z}, s_1^{k-1}r/5)$. 

Applying Lemma \ref{lem:closeness}, this means that $(U,x,\phi)$ is $\frac{\epsilon'}{4AL^2(1+2L^2)s_1}$-close to  $(\tilde{Z}, \tilde{z}, \tilde{\psi})$ at the new scale $4AL^2(1+2L^2) s_1^{k} r/10$.

By another use of Lemma \ref{lem:closeness}, we also have that $(\tilde{X},\tilde{x},\tilde{\phi})$ and $(\tilde{Z}, \tilde{z}, \tilde{\psi})$ are $2(\epsilon+\frac{\epsilon'}{4AL^2(1+2L^2)s_1})$-close at scale $8AL^2(1+2L^2) s_1^k r/10$.

By Lemma \ref{lem:semmes}, they are therefore continuously $2\Lambda(\epsilon+\frac{\epsilon'}{8AL^2(1+2L^2)s_1})$-close at scale
$$t=4AL^2(1+2L^2) s_1^k r/10.$$
Observe that $t< s_1^{k-1} r/10$ by our choice of $s_1$.

By our choice of $\epsilon$ and $\epsilon'$, we have 
$$ 2\Lambda(\epsilon+\frac{\epsilon'}{8AL^2(1+2L^2)s_1}) < 1/100A^3L^3.$$
So, to recap, $(\tilde{X},\tilde{x},\tilde{\phi})$ and $(\tilde{Z}, \tilde{z}, \tilde{\psi})$ are continuously $\delta$-close at scale $t$ (for $\delta<\frac{1}{100A^3L^3}$). Furthermore $\tilde{\psi}$ is $L$-bilipschitz on $B(\tilde{z}, s_1^{k-1} r/5) \supset B(\tilde{z},22A^2t)$ in $\tilde{Z}$.

Therefore, Lemma \ref{lem:degree1} implies that $\tilde{\phi}$ is $L$-bilipschitz on
$$ B(\tilde{x}, t/2AL^2(L+2L^2)) = B(\tilde{x}, s_1^k r/5).$$ 

Now, the fact that $(U,x,\phi)$ and $(\tilde{X},\tilde{x},\tilde{\phi})$ are $\epsilon$-close at scale $4AL^2(1+2L^2)s_1^k r/10$ implies that they are $4AL^2(1+2L^2)\epsilon$-close at scale $s_1^k r/10$. Since $4AL^2(1+2L^2)\epsilon < \epsilon'$, this completes the proof of Claim \ref{claim:iteration}.

\end{proof}

With Claim \ref{claim:iteration} proven, Lemma \ref{lem:iteration} now follows: Let $x,y$ be any points of $Y\cap B(q, r/20)$. Choose $k\geq 0$ such that
$$ s_1^{k+1} r/10 \leq d(x,y) < s_1^k r/10$$

By Claim \ref{claim:iteration}, $(U,x,\phi)$ is $\epsilon'$-close to a triple $(\tilde{X}, \tilde{x}, \tilde{\phi})$ at scale $s_1^k r/10$, for which $\tilde{\phi}$ is $L$-bilipschitz on $B(\tilde{x}, s_1^k r/5)$.

It follows that
\begin{align*}
|\phi(x) - \phi(y)| &\geq L^{-1} (d(x,y) - 2\epsilon's_1^k r/10) - 2\epsilon's_1^k r/10\\&\geq (L^{-1}-2L^{-1}\epsilon'(s_1)^{-1} - 2\epsilon'(s_1)^{-1})d(x,y)\\ 
&\geq (2L)^{-1}d(x,y).
\end{align*}

\end{proof}

We now make one further decomposition of $Y$. Since $U$ is $(C,R)$-uniformly pointwise doubling, for some $C\geq 1$ and $R>0$, there is a constant $C'=C'(C,A(\lambda),L)\geq 1$ such that if $q\in U$, $r<R$ and
$$B(q,s_0 r) \subseteq B(p,r),$$ 
then
\begin{equation}\label{eqn:Wdense}
\mu(B(p,r)) \leq C'\mu(B(q,s_0r/10)),
\end{equation} 
where $s_0=s_0(L,D,A)$ is the constant from Proposition \ref{prop:subball} which was already fixed.

For $r_1\in (0,r_0)$, let
\begin{equation}\label{eqn:Wdef}
 W_{r_1} = \left\{ p\in Y : \frac{\mu(Y\cap B(p,r))}{\mu(B(p,r))} \geq 1-(2C')^{-1} \text{ for all } r\in (0,r_1)\right\}.
\end{equation}
By the Lebesgue density theorem, $\mu$-a.e. point of $Y$ is in $W_{r_1}$ for some choice of $r_1 \in (0,r_0) \cap \{1, 1/2, 1/3, \dots\}$.

Fix $r_1\in (0,r_0)$ and let $W = W_{r_1}\subset Y$. We will show that $\phi$ is bilipschitz on a $W$-neighborhood of each point of $W$.

\begin{lemma}\label{lem:porous}
There is a constant $s_2>0$ (depending only on $D$, $A(\lambda)$, and $L$) with the following property:

Let $p$ be a point of $W$ and fix $r<r_1/2$. Then there is a point $q\in Y$ and a ball $B(q, s_2 r) \subset B(p,r)$ such that $\phi|_{B(q,s_2 r) \cap Y}$ is $2L$-bilipschitz.
\end{lemma}
\begin{proof}
By Lemma \ref{lem:rescaling}, there is an $A$-LC, $L$-LQ topological plane $(\hat{X}, \hat{p}, \hat{\phi})$ such that $(U,p,\phi)$ and $(\hat{X}, \hat{p}, \hat{\phi})$ are $\epsilon$-close at scale $r$. 

By Proposition \ref{prop:subball}, there is a ball $B(\hat{q}, s_0 r/2) \subseteq B(\hat{p}, r/2)$ on which $\hat{\phi}$ is $L$-bilipschitz.

It follows from Lemma \ref{lem:closeness} that, for some $q_0\in U$, $B(q_0, s_0r/2)\subset B(p,r)$ and in addition
$$ (\hat{X}, \hat{q}, \hat{\phi}) \text{ and } (U,q_0,\phi) \text{ are } \frac{2}{s_0}\epsilon\text{-close at scale } s_0r/2.$$ 

Since $r<r_1$ and $p\in W$, there must be a point $q\in Y \cap B(q_0, s_0r/10)$. It follows from Lemma \ref{lem:closeness} that there is a point $\hat{q}_0\in B(\hat{q}, s_0r/5)$ such that
$$ (\hat{X}, \hat{q}_0, \hat{\phi}) \text{ and } (U,q,\phi) \text{ are } \frac{10}{s_0}\epsilon\text{-close at scale } s_0r/10.$$

We know that $\hat{\phi}$ is $L$-bilipschitz on $B(\hat{q},s_0r/2)\supseteq B(\hat{q}_0, s_0r/10)$.

Therefore, by Lemma \ref{lem:iteration}, $\phi$ is $2L$-bilipschitz on $B(q, s_0r/200) \cap Y$. This completes the proof.

\end{proof}

The proof of Theorem \ref{thm:main} is now completed as follows. For each of the countably many choices of $L,\lambda\in\mathbb{N}$ and $r_0\in\{1, 1/2, 1/3, \dots\}$, we obtain a closed set $Y=Y_{L,\lambda, r_0}\subset U$ as in \eqref{eqn:Ydef}. For each further choice of $r_1\in\{1, 1/2, 1/3, \dots\}$, we obtain a set $W = W_{L,\lambda, r_0, r_1} \subseteq Y$. The union of these sets $W$ over all countably many choices of parameters covers $\mu$-almost all of $U$.

Consider the following relatively open subset $W'\subset W$:
$$ W' = \{ p\in W\subset Y:  \phi \text{ is bilipschitz on } B(p,t) \cap Y \text{ for some } t>0\}.$$
By Lemma \ref{lem:porous}, the set $W\setminus W'$ is porous in $Y$, hence porous in $U$. Because $U$ is uniformly pointwise doubling, it follows that the set $W\setminus W'$ has $\mu$-measure zero. (See Remark \ref{rmk:ptwiseporous}.) Then
\begin{align*}
W' &\subseteq \bigcup_{j\in\mathbb{N}} \{p\in W': \phi|_{B(p,1/j)\cap Y} \text{ is bilipschitz}\}\\
&\subseteq \bigcup_{j\in\mathbb{N}} \bigcup_{i\in \mathbb{N}} \left(B(p^j_i,1/j) \cap Y\right),
\end{align*}
where $\{p^j_i\}_{i=1}^\infty$ is any countable dense subset of the set 
$$\{p\in W': \phi|_{B(p,1/j)\cap Y} \text{ is bilipschitz}\}.$$

This shows that $W'$ is covered by countably many Borel sets $B(p^j_i,1/j) \cap Y$ on which $\phi$ is bilipschitz. Since $\mu(W\setminus W') = 0$, it follows that $\mu|_W$ is $2$-rectifiable. Writing $U$, up to measure zero, as a countable union of sets $W$ as above, we have shown that $\mu|_U$ is $2$-rectifiable. This completes the proof of Theorem \ref{thm:main}.
\end{proof}

\begin{proof}[Proof of Corollary \ref{cor:main}]
We can now establish Corollary \ref{cor:main} as follows. Since $X$ is metrically doubling, $\mu$ is pointwise doubling, and $\mu$ assigns measure zero to porous sets in $X$, we see that the blowups of $U$ and the blowups of $X$ coincide at almost every point of $U$. (See Remark \ref{rmk:ptwiseblowup}.) As in the proof of Lemma \ref{lem:LQlimit}, every blowup of $X$ is a topological plane. Hence, by Theorem \ref{thm:main}, $n\leq 2$.

If $n=2$, then $\mu|_U$ is $2$-rectifiable, again by Theorem \ref{thm:main}.
\end{proof}

\begin{rmk}
It is straightforward to see that the same broad outline, much simplified, can be used to show the $1$-dimensional analog of Theorem \ref{thm:main}: Let $(X,d,\mu)$ be a pointwise doubling space and $U\subset X$ be a Borel subset such that, for $\mu$-a.e. $x\in U$, each blowup of $X$ at $x$ is homeomorphic to $\RR$. If $\mu|_U$ has $n$ $\phi$-independent Alberti representations for $\phi:X\rightarrow\RR^n$ Lipschitz, then $n\leq 1$, with equality only if $\mu|_U$ is $1$-rectifiable.

Indeed, in this case, the blowups of the mapping $\phi$ at generic points are globally bilipschitz (moreover, affine), and the analog of Lemma \ref{lem:rescaling} essentially yields $1$-rectifiability.
\end{rmk}

\section{Examples}\label{sec:examples}

We first note a simple example which shows that rectifiability does not follow from simply assuming that $(X,d,\mu)$ is a doubling, LLC, topological surface, even if it supports a single Alberti representation.
\begin{example} \label{example1}

Let $X= \RR \times Y$, where $Y$ is the ``snowflaked'' metric space $(\RR, |\cdot|^{1/2})$. Then $(X,d,\mathcal{H}^3)$ is a doubling metric measure space which is also an LLC topological surface. Furthermore, the restriction of $\mu$ to every compact subset of $X$ supports one Alberti representation, simply given by Fubini's theorem in the $\RR$ factor.

On the other hand, no Lipschitz map from a compact set in $\RR^2$ can have an image of positive $\mathcal{H}^3$-measure in any metric space, and so $\mathcal{H}^3|_U$ is not $2$-rectifiable for any $U\subset X$ of positive measure. Moreover, the space $X$ is \textit{purely $2$-unrectifiable}, in the sense that $\mathcal{H}^2(f(E))=0$ for every compact $E\subset \RR^2$ and Lipschitz $f:E\rightarrow X$.
\end{example}

For a related and more interesting example, see the appendix by Schul and Wenger in \cite{SW10}.

The next two examples show that, in the absence of a quantitative topological assumption, such as LLC or having blowups that are topological planes, either part of Theorem \ref{thm:main} or Corollary \ref{cor:main} may fail, even if $(X,d,\mu)$ is a pointwise doubling topological surface supporting multiple Alberti representations.

\begin{example}\label{example2}
Let $C$ be a Jordan curve in the plane (homeomorphic to the circle) of positive two-dimensional Lebesgue measure. In fact, by a construction of Sierpi\'nski-Knopp (see \cite{Sa94}, Section 8.3), we can ensure that the restriction of Lebesgue measure $\mathcal{L}^2$ to $C$ is Ahlfors $2$-regular, i.e., satisfies
$$ M^{-1}r^2 \leq \mathcal{L}^2(C \cap B(x,r)) \leq Mr^2$$
for some $M\geq 1$, all $x\in X$, and all $r\in (0,1)$.

Let $X = C \times \RR$ in $\RR^3$, which is a topological surface. Equip $X$ with the restriction of the distance $|\cdot|$ from $\RR^3$ and with the restriction of $3$-dimensional Lebesgue measure, which is doubling on $X$. Then, as a positive measure set in $\RR^3$, $X$ is a Lipschitz differentiability space of dimension $3$, in the sense of \cite{Ba15}. In particular, there are Borel sets $U_i\subset X$ and Lipschitz maps $\phi_i:X\rightarrow \RR^3$ such that $\mu(X\setminus U_i)=0$ and $\mu|_{U_i}$ supports three $\phi_i$-independent Alberti representations for each $i$ (see \cite{Ba15}, Theorem 6.6.).

Thus, the upper bound $n\leq 2$ on the number of independent Alberti representations in Theorem \ref{thm:main} may fail in the absence of the assumption on blowups, and the upper bound in Corollary \ref{cor:main} may fail in the absence of the LLC assumption. 
\end{example}

\begin{example}\label{example3}
Consider the same topological surface $X$ in $\RR^3$ as in the previous example, but now consider $\RR^3$ equipped with the Heisenberg group metric $d_{\mathbb{H}}$. Endow $X$ with the restriction of $d_{\mathbb{H}}$ and the restriction of $3$-dimensional Lebesgue measure, which is the same (up to constant factors) as $\mathcal{H}^4$ in the Heisenberg group. The space $(X,d_\mathbb{H},\mathcal{H}^4)$ is pointwise doubling by the Lebesgue density theorem in the Heisenberg group, and porous subsets of $X$, being also porous subsets of the Heisenberg group, have $\mathcal{H}^4$-measure zero.

As a positive measure set in the Heisenberg group, $(X,d_\mathbb{H},\mathcal{H}^4)$ is also a Lipschitz differentiability space with $X$ itself a chart of dimension $2$. Thus, as in the previous example, it admits a Borel decomposition into $U_i$ such that each $\mu|_{U_i}$ supports two $\phi_i$-independent Alberti representations, for some Lipschitz $\phi_i:X\rightarrow\RR^2$.

However, no Lipschitz map from a compact set in $\RR^2$ can have an image of positive $\mathcal{H}^4$-measure in any metric space, and so $\mathcal{H}^4|_U$ cannot be $2$-rectifiable for any $U\subset X$ of positive measure. (In fact, as in Example \ref{example1}, more is true here: $(X,d_\mathbb{H})$ is purely $2$-unrectifiable, as a consequence of the pure $2$-unrectifiability of the Heisenberg group (\cite{AK00}, Theorem 7.2).)

Thus, the conclusion of $2$-rectifiability in the case of equality may fail in Theorem \ref{thm:main} in the absence of the assumption on blowups, and in Corollary \ref{cor:main} in the absence of the LLC assumption. 
\end{example}

\begin{appendix}
\section{Blowups and annular linear connectivity}\label{sec:appendix}
Recall the notion of annular linear connectivity (ALC) from Definition \ref{def:LLCALC}. The goal of this appendix is to prove the following proposition, which allows a self-strengthening of the hypotheses in Theorem \ref{thm:main}.

We will use the notion of a \textit{cut point} $y$ in a connected space $Y$: a point such that $Y\setminus \{y\}$ is disconnected.

\begin{prop}\label{prop:blowupsALC}
Let $(X,d,\mu)$ be complete and metrically doubling with $\mu$ pointwise doubling. Let $U\subset X$ be a Borel subset such that, for $\mu$-a.e. $x\in U$, each blowup of $X$ at $x$ is connected and has no cut points.

Then for $\mu$-a.e. $x\in U$, there is a constant $\lambda = \lambda(x)$ such that each blowup of $X$ at $x$ is $\lambda$-ALC.
\end{prop}
In particular, Proposition \ref{prop:blowupsALC} applies when the blowups of $X$ at almost every point of $U$ are homeomorphic to $\RR^2$.

The following preliminary definition will be useful.
\begin{definition}
We call a metric space $X$ \textit{linearly connected} if there is a constant $L\geq 1$ such that, for all $x,y\in X$, there is a compact, connected set containing $x$ and $y$ of diameter at most $Ld(x,y)$.
\end{definition}

\begin{lemma}\label{lem:ALC}
Let $\mathcal{C}$ be a collection of complete, metrically $D$-doubling metric spaces with the following property: For each sequence $\{r_k\}$ of positive real numbers and each sequence $\{(X_k, p_k)\}$ such that $X_k \in \mathcal{C}$, $p_k\in X_k$, and $\{r_k^{-1} X_k, p_k\}$ converges in the pointed Gromov-Hausdorff sense, the limit is connected and has no cut points.

Then there is a constant $\lambda$ such that all elements of $\mathcal{C}$ are $\lambda$-ALC.
\end{lemma}

Note that the hypotheses of Lemma \ref{lem:ALC} include the assumption that each element of $\mathcal{C}$ is itself is connected with no cut points.

We now explain how to prove Proposition \ref{prop:blowupsALC} given Lemma \ref{lem:ALC}.

\begin{proof}[Proof of Proposition \ref{prop:blowupsALC}]
We may assume, by Lemma \ref{lem:ptwisedoubling} and Remark \ref{rmk:ptwiseblowup}, that $U$ is complete, metrically $D$-doubling, and $(C,R)$-uniformly doubling, for constants $D\geq 1$, $C\geq 1$, $R>0$. Then, for a.e. $x\in U$, the blowups of $X$ at $x$ are the blowups of $U$ at $x$.

For each $x\in U$, let $\mathcal{B}_x$ denote the collection of all pointed metric spaces $(Z,p)$ that arise as blowups of $U$ at $x$. The collection $\mathcal{B}_x$ is closed under pointed Gromov-Hausdorff convergence.

Theorem 1.1 of \cite{Le11} shows that, for a.e. $x\in U$, if $(Z,p)\in \mathcal{B}_x$ and $q\in Z$, then $(Z,q)\in \mathcal{B}_x$. (Note that, although Theorem 1.1 of \cite{Le11} is stated for doubling measures, the proof relies only on the estimates provided by the fact that $U$ is $(C,R)$-uniformly doubling. This was also noted in Section 9 of \cite{BL15_RNP}.)

Let $\mathcal{C}_x$ be the collection of (unpointed)  metric spaces $Z$ that arise as blowups of $U$ at $x$. It follows from the previous paragraph (and the fact that rescalings of blowups are blowups) that, for a.e. $x\in U$, $(r^{-1}Z,p)$ is in $\mathcal{B}_x$ for all $Z\in\mathcal{C}_x$, $p\in Z$, and $r >0$.

Hence, any pointed Gromov-Hausdorff limit of rescaled pointed elements of $\mathcal{C}_x$ as in Lemma \ref{lem:ALC} is an element of $\mathcal{B}_x$. By our assumption on $U$, such a limit must be connected with no cut points. 

Thus, for a.e. $x\in U$, the collection $\mathcal{C}_x$ satisfies the hypotheses of Lemma \ref{lem:ALC}. It follows that there is a constant $\lambda=\lambda(x)$ such that each element of $\mathcal{C}_x$ (in particular, each blowup of $U$ at $x$) is $\lambda$-ALC. This completes the proof of Proposition \ref{prop:blowupsALC}.
\end{proof}

It remains to prove Lemma \ref{lem:ALC}. To do so, we will use the following Lemma, which is a minor modification of Proposition 5.4 of \cite{Ki15}. 

\begin{lemma}\label{lem:lincon}
Let $\mathcal{C}$ be a collection of metric spaces satisfying the hypotheses of Proposition \ref{lem:ALC}. Then there is a constant $L\geq 1$ such that each element of $\mathcal{C}$ is linearly connected with constant $L$.
\end{lemma}
\begin{proof}
This is proven in Proposition 5.4 of \cite{Ki15}, in the case where $\mathcal{C}$ has a single element $X$ (in which case the Gromov-Hausdorff limits of pointed rescalings of $X$ are called ``weak tangents'' of $X$). However, an identical proof works under our assumption that all elements of $\mathcal{C}$ are metrically $D$-doubling, since the same compactness argument can be run. (Note that the boundedness assumption in Proposition 5.4 of \cite{Ki15} is not needed here, because we allow arbitrary scalings in the hypotheses of Proposition \ref{lem:ALC}.)
\end{proof}

\begin{lemma}\label{lem:discreteLLC}
Let $X$ be a $L$-linearly connected metric space that has the following property, for some $\mu\geq 1$:

For all $p\in X$ and $r\in (0, \diam(X)]$, and for all $x,y\in A(p, r, 2r)$, there is a finite set
$$P=\{x_0, x_1, \dots, x_n\} \subset A(p, r/\mu, 2\mu r)$$
such that
\begin{equation}\label{eqn:discreteLLC1}
 x_0 = x \text{ and } x_n = y,
\end{equation}
and
\begin{equation}\label{eqn:discreteLLC2}
d(x_i, x_{i+1}) \leq \frac{1}{2L} \dist(P, p) \text{ for each } i\in\{0,\dots, n-1\}.
\end{equation}

Then $X$ is $\lambda$-ALC, where $\lambda$ depends only on $\mu$.
\end{lemma}
\begin{proof}
Consider any $p\in X$, $r\in (0, \diam(X)]$, and $x,y\in A(p, r, 2r)$. Let $P= \{x_0, \dots, x_n\}\subset A(p,r/\mu, 2\mu r)$ satisfy \eqref{eqn:discreteLLC1} and \eqref{eqn:discreteLLC2}.

For each $i\in \{0, \dots, n-1\}$ we use the linear connectedness of $X$ to join $x_i$ to $x_{i+1}$ by a continuum of diameter at most
$$ Ld(x_i, x_{i+1}) \leq \frac{1}{2}\dist(P,p). $$
The union of these continua forms a  continuum joining $x$ to $y$ inside
$$ A(p, r/2\mu, 3\mu r),$$
which proves the lemma with $\lambda = 2\mu$.
\end{proof}

\begin{proof}[Proof of Lemma \ref{lem:ALC}]
Let $\mathcal{C}$ be a collection of metric spaces satisfying the hypotheses of Lemma \ref{lem:ALC}. By Lemma \ref{lem:lincon}, we immediately have a constant $L\geq 1$ such that each $X\in \mathcal{C}$ is linearly connected with constant $L$. Note that this immediately implies that any pointed Gromov-Hausdorff limit of rescaled, pointed elements of $\mathcal{C}$ is also linearly connected.

Therefore, to show that all $X\in \mathcal{C}$ are uniformly ALC, we need only verify the existence of a constant $\mu \geq 1$ such that each $X\in\mathcal{C}$ satisfies the hypotheses of Lemma \ref{lem:discreteLLC} with constant $\mu$.

Suppose that there is no such constant $\mu$. Then for all $k\in \mathbb{N}$, there is a space $X_k\in \mathcal{C}$, a point $p_k\in X_k$, a radius $r_k>0$, and points $x_k, y_k \in A(p_k, r_k, 2r_k)$, such that there is no finite set 
$$P= \{z_0=x_k, z_1,  \dots, z_{m-1}, z_m=y_k\},$$
contained in $A(p_k, \frac{1}{k}r_k, 2kr_k)$ and satisfying
\begin{equation}\label{eqn:discretechain}
d(z_i, z_{i+1}) \leq \frac{1}{2L} \dist(P, p_k) \text{ for all } i\in\{0,m-1\}
\end{equation}

Consider the uniformly doubling sequence of pointed metric spaces $\{(Y_k,p_k):=(r_k^{-1}X, p_k)\}$. Let $(Y_\infty, p_\infty)$ be a pointed Gromov-Hausdorff limit of a subsequence of this sequence, which for convenience we continue to label with the index $k$.

For all $\epsilon>0$, there exists $K\in \mathbb{N}$ such that, for all $k\geq K$, there are $\epsilon$-isometries
\begin{equation}\label{eqn:isom1}
f_k:B_{Y_k}\left(p_k, \frac{1}{\epsilon}\right)\rightarrow Y_\infty \text{ and } g_k:B_{Y_\infty}\left(p_\infty, \frac{1}{\epsilon}\right)\rightarrow Y_k
\end{equation}
such that
\begin{equation}\label{eqn:isom2}
d_{Y_\infty}(f_k(g_k(x)), x) \leq \epsilon \text{ for all } x\in B_{Y_\infty}(p_\infty, 1/2\epsilon)
\end{equation}
and
\begin{equation}\label{eqn:isom3}
 d_{Y_\infty}(f_k(p_k), p_\infty) \leq \epsilon \text{ and } d_{Y_k}(g_k(p_\infty), p_k) \leq \epsilon
\end{equation}

For all $k$ sufficiently large, the points $f_k(x_k)$ and $f_k(y_k)$ all lie in $B_{Y_\infty}(p_\infty, 3)$. By passing to a further subsequence if necessary, we may therefore also assume that $f_k(x_k)$ and $f_k(y_k)$ converge to points $x_\infty$ and $y_\infty$, respectively, in $\overline{B}(p_\infty, 3) \in Y_\infty$.

The space $Y_\infty$ is a pointed Gromov-Hausdorff limit of pointed rescalings of elements of $\mathcal{C}$. Hence, by assumption, it is connected with no cut points. Furthermore, as remarked at the beginning of this proof, it is linearly connected. A simple connectedness argument then yields a compact connected set $C_\infty \subset Y_\infty\setminus \{p_\infty\}$ containing both $x_\infty$ and $y_\infty$. (Indeed, the set of $y\in Y_\infty\setminus \{p_\infty\}$ that can be joined to $x_\infty$ by such a continuum is open in $Y_\infty\setminus \{p_\infty\}$, as is its complement, by the linear connectedness of $Y_\infty$.)

For some choice of $0<r<1<R<\infty$, $C_\infty$ must lie in $A(p_\infty, r, R)$. Let 
$$\epsilon = \min(r/100L, 1/100R).$$
There is a finite set $P_\infty = \{x_\infty^0, x_\infty^1, \dots, x_\infty^n\} \subseteq C_\infty$ such that $x_\infty^0=x_\infty$, $x_\infty^n=y_\infty$, and
$$ d(x_\infty^i, x_\infty^{i+1}) \leq \epsilon \text{ for each } i\in \{0, \dots, n-1\}.$$

Choose $k>2\epsilon^{-1}$ sufficiently large so that there are $\epsilon$-isometries $f_k$ and $g_k$ as in \eqref{eqn:isom1}, \eqref{eqn:isom2}, and \eqref{eqn:isom3}. We can also assume that $k$ is large enough so that
$$ d_{Y_\infty}(f_k(x_k), x_\infty) \leq \epsilon \text{ and } d_{Y_\infty}(f_k(y_k), y_\infty) \leq \epsilon.$$

Let $Q_k\subset Y_k$ denote the set
$$ Q_k = \{ z_0 = x_k, z_1 = g_k(x^0_\infty), z_2 = g_k(x^1_\infty), \dots, z_n = g_k(x^n_\infty), z_{n+1} = y_k \}.$$
Because $g_k$ is an $\epsilon$-isometry into $Y_k = r_k^{-1} X_k$, and because of equation \eqref{eqn:isom2}, we see that
$$ d_{X_k}(z_i, z_{i+1}) \leq 3\epsilon r_k \text{ for each } i\in \{0, \dots, n\}.$$
Furthermore,
$$\dist_{X_k}(Q_k, p_k) \geq r r_k  - 2\epsilon r_k \geq \frac{1}{2}r r_k.$$
and
$$ Q_k \subset B_{X_k}(p_k, (R+ 2\epsilon)r_k) \subset B_{X_k}(p_k, 2R r_k).$$

Thus, $Q_k = \{z_0=x_k, z_1, \dots, z_{n+1}=y_k\}$ is contained in
$$A_{X_k}(p_k, \frac{1}{2}r r_k, 2R r_k) \subset A_{X_k}(p_k, \frac{1}{k}r_k, 2kr_k),$$
where this last inclusion follows from our assumption that 
$$k\geq 2\epsilon^{-1} \geq \max\{100/r, 100R\}.$$

In addition,
$$ d_{X_k}(z_i, z_{i+1}) \leq 3\epsilon r_k  \leq \frac{1}{2L} \frac{1}{2}r r_k \leq \frac{1}{2L} \dist_{X_k}(Q_k, p_k) \text{ for each } i\in \{0, \dots, n\}.$$

Hence, $Q_k$ is contained in $A(p_k, r_k/k, 2kr_k)$ and satisfies \eqref{eqn:discretechain}. This is a contradiction.
\end{proof}
\end{appendix}

\bibliography{rectifiable_planes}
\bibliographystyle{alpha}
\end{document}